\documentclass[a4paper,12pt]{article}

\usepackage{amsmath}
\usepackage{amsthm}
\newtheorem{theorem}{Theorem}[section]

\newtheorem{proposition}{Proposition}[section]
\newtheorem{lemma}{Lemma}[section]
\newtheorem{remark}{Remark}[section]

\usepackage{newtxtext}
\usepackage{newtxmath}
\usepackage{mathtools}
\mathtoolsset{showonlyrefs=true}
\usepackage{algorithm}
\usepackage{algpseudocode}
\algrenewcommand\textproc{\texttt}
\usepackage[truedimen, margin=20truemm]{geometry}

\usepackage{hyperref}
\hypersetup{%
    pdftitle={Computing the matrix fractional power with the double exponential formula},
    pdfauthor={F. Tatsuoka, T. Sogabe, Y. Miyatake, T. Kemmochi, S.-L. Zhang},
    pdfkeywords={matrix funciton, matrix fractional power, numerical quadrature, double exponential formula},
    colorlinks, citecolor=blue
}

\title{Computing the Matrix Fractional Power with the Double Exponential Formula\thanks{%
This paper is an augmented version of our previous report (F. Tatsuoka, T. Sogabe, Y. Miyatake, and S.-L. Zhang, A note on computing the matrix fractional power using the double exponential formula, Trans. Jpn. Soc. Ind. Appl. Math., 28(2018), pp. 142--161), which is in Japanese.
There are two differences between this paper and the previous report.
One is that this paper presents the convergence rate analysis of the double exponential (DE) formula, while the previous report does not.
The other is that this work employs a new transformation $t(x) = \exp(\alpha\pi\sinh(x)/2)$, while the previous work used the transformation $t(x) = \exp(\alpha\sinh(x))$.
Several numerical results have shown that the transformation employed in this work is better than that in the previous work in terms of the convergence speed.
}}

\author{%
  Fuminori Tatsuoka\footnotemark[2]
  \and Tomohiro Sogabe\footnotemark[2]
  \and Yuto Miyatake\footnotemark[3]
  \and Tomoya Kemmochi\footnotemark[2]
  \and Shao-Liang Zhang\footnotemark[2]
}

\newcommand{\bbC}{\mathbb{C}}
\newcommand{\bbN}{\mathbb{N}}
\newcommand{\bbR}{\mathbb{R}}

\newcommand{\calD}{\mathcal{D}}
\newcommand{\calI}{\mathcal{I}}
\newcommand{\calL}{\mathcal{L}}
\newcommand{\calO}{\mathcal{O}}
\newcommand{\bmb}{\boldsymbol{b}}

\newcommand{\bmx}{\boldsymbol{x}}

\newcommand{\rme}{\mathrm{e}}
\newcommand{\rmi}{\mathrm{i}}
\newcommand{\rmW}{\mathrm{W}}
\newcommand{\rmIm}{\mathrm{Im}}
\newcommand{\dd}[1]{\,\mathrm{d} #1}
\newcommand{\Fde}{F_{\mathrm{DE}}}
\newcommand{\fde}{f_{\mathrm{DE}}}

\newcommand{\norm}[1]{\left\lVert #1 \right\rVert}

\newcommand{\tildeFde}{\tilde{F}_{\mathrm{DE}}}
\newcommand{\lambdamax}{\lambda_{\mathrm{max}}}
\newcommand{\lambdamin}{\lambda_{\mathrm{min}}}
\newcommand{\mumax}{\mu_{\mathrm{max}}}
\newcommand{\mumin}{\mu_{\mathrm{min}}}
\newcommand{\sigmamax}{\sigma_{\mathrm{max}}}
\newcommand{\sigmamin}{\sigma_{\mathrm{min}}}
\newcommand{\asinh}{\mathrm{asinh}}

\begin{document}

\maketitle

\renewcommand{\thefootnote}{\fnsymbol{footnote}}
\footnotetext[2]{Department of Applied Physics, Graduate School of Engineering, Nagoya University, Furo-cho, Chikusa-ku, Nagoya 464-8603, Japan (\texttt{\{f-tatsuoka,sogabe,kemmochi,zhang\}@na.nuap.nagoya-u.ac.jp}).}
\footnotetext[3]{Cybermedia Center, Osaka University, 1-32 Machikaneyama, Toyonaka, Osaka 560-0043, Japan (\texttt{miyatake@cas.cmc.osaka-u.ac.jp}).}

\renewcommand{\thefootnote}{\arabic{footnote}}
\begin{abstract}
  Two quadrature-based algorithms for computing the matrix fractional power $A^\alpha$ are presented in this paper.
  These algorithms are based on the double exponential (DE) formula, which is well-known for its effectiveness in computing improper integrals as well as in treating nearly arbitrary endpoint singularities.
  The DE formula transforms a given integral into another integral that is suited for the trapezoidal rule; in this process, the integral interval is transformed to the infinite interval.
  Therefore, it is necessary to truncate the infinite interval into an appropriate finite interval.
  In this paper, a truncation method, which is based on a truncation error analysis specialized to the computation of $A^\alpha$, is proposed.
  Then, two algorithms are presented--- one computes $A^\alpha$ with a fixed number of abscissas, and the other computes $A^\alpha$ adaptively.
  Subsequently, the convergence rate of the DE formula for Hermitian positive definite matrices is analyzed.
  The convergence rate analysis shows that the DE formula converges faster than the Gaussian quadrature when $A$ is ill-conditioned and $\alpha$ is a non-unit fraction.
  Numerical results show that our algorithms achieved the required accuracy and were faster than other algorithms in several situations.
\end{abstract}

\section{Introduction}\label{sec:introduction}
We consider the computation of the matrix power $A^\alpha$ for a matrix $A\in\bbC^{n\times n}$ and a real number $\alpha\in\bbR$.
The matrix exponential is defined as $\exp(X) = I + X + X^2/2! + X^3/3! + \cdots$ $(X \in \bbC^{n\times n})$, and the matrix logarithm is defined as an inverse function of the matrix exponential, i.e., any matrix $Y\in\bbC^{n\times n}$ satisfying $\exp(Y) = A$ is a logarithm of $A$.
If all the eigenvalues of $Y$ lie in the strip $\{z\in\bbC \colon |\mathrm{Im}(z)| < \pi\}$, then $Y$ is called the principal logarithm and is denoted by $\log(A)$.
Then, the principal matrix power $A^\alpha$ is defined as follows:
\begin{align}
  A^\alpha = \exp(\alpha\log(A)).
\end{align}
It is known that $A^\alpha$ exists and is unique when all the eigenvalues of $A$ lie in the set $\{z\in\bbC\colon z\notin (-\infty,0]\}$; see, e.g., \cite{higham_schur-pade_2011}.
The matrix fractional power arises in several situations of computational science, e.g., fractional differential equations \cite{aceto_rational_2017,burrage_efficient_2012,szekeres_finite_2017}, lattice QCD calculation \cite{aoki_exact_2003}, and computation of the weighted matrix geometric mean \cite{fasi_computing_2018}.

In this study, we computed $A^\alpha$ using the double exponential (DE) formula \cite{takahasi_double_1974} for the integral representation:
\begin{align}\label{eq:ir_realaxis}
  A^\alpha = \frac{\sin(\alpha\pi)}{\alpha\pi}A \int_0^\infty (t^{1/\alpha}I+A)^{-1} \dd{t}
  \qquad (0<\alpha<1),
\end{align}
see, e.g., \cite[Eq. (1.4)]{higham_schur-pade_2011}.
Without loss of generality, we assume $0 < \alpha< 1$ in this work because for any $\alpha\in\bbR$, we can compute $A^\alpha =A^{\lfloor\alpha\rfloor}A^{\lfloor\alpha\rfloor - \alpha}$, where $\lfloor\alpha\rfloor\in\bbN$ and $\lfloor\alpha\rfloor - \alpha \in [0,1)$.
The DE formula transforms the integral \eqref{eq:ir_realaxis} into another integral suited for the trapezoidal rule.
In particular, the transformed integrand decays double exponentially, and the transformed interval is the infinite interval $(-\infty,\infty)$, see, Section 2.
Hence, it is necessary to truncate the infinite interval into a finite interval appropriately to utilize the DE formula.

In this paper, we first present an algorithm for computing $A^\alpha$ using the $m$-point DE formula.
To develop this algorithm, we analyzed the truncation error of the DE formula for \eqref{eq:ir_realaxis} and proposed a method for truncating the infinite interval into a finite interval within a given tolerance.
Then, we present an adaptive quadrature algorithm based on the first algorithm to achieve the required accuracy.
This adaptive quadrature algorithm employs an a posteriori error estimation technique discussed in \cite{cardoso_computation_2012}.
Subsequently, we analyzed the convergence rate of the DE formula for Hermitian positive definite (HPD) matrices and compared it with that of Gaussian quadrature.
The convergence rate analysis gives us an a prior estimate of the discretization error as Gaussian quadrature in \cite{fasi_computing_2018}.
Hence, for HPD matrices, the number of abscissas can be estimated on the basis of the error estimation.

Several computational methods for $A^\alpha$ have been proposed. Examples are the Schur--Pad\'e algorithm \cite{higham_schur-pade_2011,higham_improved_2013}, the Schur logarithmic algorithm \cite{iannazzo_schur_2013}, and quadrature-based algorithms \cite{aceto_rational_2017,cardoso_computation_2012,fasi_computing_2018,hale_computing_2008}.
Among them, quadrature-based algorithms have two advantages.
One advantage is that these algorithms can compute $A^\alpha\bmb$ $(\bmb\in\bbR^n)$ without computing $A^\alpha$ itself when $A$ is large and either sparse or structured.
The other is that these algorithms can be parallelized easily in the sense that the integrand can be computed independently on each abscissa.
See, \cite[Sect. 18]{trefethen_exponentially_2014} for more details.

Quadrature-based algorithms were developed based on the two integral representations of $A^\alpha$.
One representation is based on the Cauchy integral \cite[Eq.\ (1.1)]{higham_schur-pade_2011}, which has been used as a basis for several algorithms \cite{hale_computing_2008}.
These algorithms are specialized for the case where all the eigenvalues of $A$ lie on or near the positive real axis.
Their convergence rate is better than that of other quadrature-based algorithms, but they require complex arithmetic even if $A$ is a real matrix.
To the best of our knowledge, for a general matrix $A$, no algorithms based on the Cauchy integral are available because of the difficulty in selecting the integral path.

The other integral representation is \eqref{eq:ir_realaxis}.
Unlike the Cauchy integral, the integral representation \eqref{eq:ir_realaxis} allows us to avoid selecting the integral path and avoid complex arithmetic for real matrices.
In \cite{aceto_rational_2017,cardoso_computation_2012,fasi_computing_2018}, the integral \eqref{eq:ir_realaxis} was transformed into other integrals on $[-1,1]$, and the transformed integrals were computed with Gaussian quadrature.
For example, transformations $t(u)=(1+u)/(1-u)$ and $t(v)=(1-v)^{\alpha}/(1+v)^{\alpha}$ were considered.

There are two motivations for considering the DE formula.
First, the DE formula may converge faster than Gaussian quadrature. The DE formula is known for its effectiveness in dealing with integrals on (half) infinite intervals as well as integrals with endpoint singularities, see, e.g., \cite{mori_discovery_2005,trefethen_exponentially_2014}. Hence, the DE formula is often used when the convergence of Gaussian quadrature is slow.
Indeed, the convergence of Gaussian quadrature for \eqref{eq:ir_realaxis} will be slow when the condition number of $A$ is large and $\alpha$ is a non-unit fraction.
This slow convergence is indicated by our numerical results and analysis for HPD matrices in \cite{fasi_computing_2018}.

Second, the DE formula can be used as an efficient adaptive quadrature algorithm.
One of the problems of Gaussian quadrature is that it incurs high computational cost for an a posteriori error estimate.
For HPD matrices, one can use an a prior estimate technique that computes $\lambda^\alpha$ for $\lambda$ as the extreme eigenvalues of $A$.
This estimation incurs lower cost than the computation of the integrand in \eqref{eq:ir_realaxis}, see, \cite{fasi_computing_2018}.
For general matrices, one can use an a posteriori estimate technique that compares two approximations of $A^\alpha$ with a different number of abscissas \cite{cardoso_computation_2012}.
However, this estimation incurs high computational cost for Gaussian quadrature because the integrand must be computed on all the abscissas when computing a new approximation with different number of abscissas.
In contrast to Gaussian quadrature, the DE formula employs the trapezoidal rule.
The trapezoidal rule can reuse an approximation of mesh size $h$ for computing a new approximation of mesh size $h/2$.
Hence, the DE formula reduces the computational cost.

The rest of this paper is organized as follows.
In the next section, a brief review of the DE formula is given.
In Section 3, the truncation error of the DE formula for $A^\alpha$ is analyzed, and the truncation method and the algorithms are presented.
In Section 4, the convergence rate of the DE formula is analyzed and compared with that of Gaussian quadrature.
Numerical results are shown in Section 5, and conclusions are given in Section 6.

\textbf{Notation:}
Unless otherwise stated, throughout this paper, $\|\cdot\|$ denotes a consistent matrix norm, e.g., $p$-norm and Frobenius norm, while $\|\cdot\|_2$ represents the 2-norm.
The symbol $\kappa(A)$ indicates the condition number, i.e., $\kappa(A) = \|A\|_2\|A^{-1}\|_2$, and $\rho(A)$ is the spectral radius of $A$, i.e., $|\lambdamax|$ for the largest absolute eigenvalue of $A$, $\lambdamax$.
The inverse of $\sinh$ is denoted by $\asinh$.
The strip of width $2d$ about the real axis is denoted by $\calD_d\coloneqq \{z\in\bbC \colon |\rmIm(z)|<d\}$.

\section{DE formula}
The DE formula exploits the fact that the trapezoidal rule for the integrals of analytic functions on the real line converges exponentially---see, e.g., \cite{mori_discovery_2005,takahasi_double_1974}.
The procedure based on the DE formula for a given integral $\int_a^b f(t) \dd{t}$, where $f$ is  a scalar function, is as follows.
\begin{enumerate}
  \item Apply a change of variables $t=t(x)$ to the integral, where the transformed interval is the infinite interval $(-\infty,\infty)$ and the transformed integrand $t'(x) f(t(x))$ decays  double exponentially as $x\to\pm\infty$.\footnote{%
    More strictly, the expression ``decays double exponentially'' means that $t'(x) f(t(x))$ is a function in the Hardy space $\mathrm{H}^\infty(\calD_d, \omega)$, where $\omega$ is a function satisfying the conditions in \cite[Thm.\ 3.2]{sugihara_optimality_1997} and $\mathrm{H}^\infty(\calD_d, \omega)$ is a function space defined in \cite[Sect. 2]{sugihara_optimality_1997}.
  }
  \item Truncate the infinite interval into a finite interval $[l,r]$.
  \item Compute $\int_l^r t'(x)f(t(x)) \dd{x}$ by the trapezoidal rule.
\end{enumerate}

If $f$ is analytic on $(a,b)$ and the change of variables is appropriate, the transformed integrand $t'(x)f(x)$ is analytic on $(-\infty,\infty)$ even if $f(t)$ is not analytic at $t=a,b$.
Thus, the DE formula can treat nearly arbitrary endpoint singularities.
In addition, the transformation with double exponential decay of the transformed integrand has a certain optimality property, see, e.g., \cite{mori_discovery_2005,sugihara_optimality_1997}.
Therefore, the DE formula can be used for computing integrals on the (half) infinite interval.

For the integral \eqref{eq:ir_realaxis}, one can apply the double exponential formula.
The integrand in \eqref{eq:ir_realaxis} is a matrix function of $A$ in the sense of \cite[Def.\ 1.2]{higham_functions_2008}, and therefore we have the explicit expressions of each element.
Considering the Jordan decomposition, each element of the integrand is 0 or in the form $\frac{\partial^k}{\partial t^k} (t^\alpha + \lambda)^{-1}$ for $k \ge 0$ and $\lambda$ being the eigenvalues of $A$.
Thus, with transformations such as $t(x) = \exp(\alpha\pi\sinh(x)/2)$, each element of the integrand in \eqref{eq:ir_realaxis} decays double exponentially as $x\to\infty$, and one can compute integrals simultaneously.

\section{Algorithms based on the DE formula}\label{sec:2:algorithms}
This section presents two algorithms for computing $A^\alpha$ based on the DE formula.
Herein, the change of variables $t(x) = \exp(\alpha\pi\sinh(x)/2)$ is considered for the DE formula.
Thus, the algorithms are based on the integral representation as follows:
\begin{align}
  A^\alpha = \int_{-\infty}^\infty \Fde(x) \dd{x},
\end{align}
where
\begin{align}\label{eq:def_f_fde}
  \Fde(x) = t'(x)F(t(x)),
  \qquad
  F(t) = \frac{\sin(\alpha\pi)}{\alpha\pi}A (t^{1/\alpha}I+A)^{-1}.
\end{align}
The first algorithm computes $A^\alpha$ on the basis of the $m$-point DE formula.
To control the truncation error, the algorithm selects a finite interval $[l,r]$ satisfying
\begin{align}\label{eq:de_err_goal}
  \left\lVert
    \int_{-\infty}^{\infty} \Fde(x) \dd{x}
    - \int_l^r \Fde(x) \dd{x}
  \right\rVert_2
  \le \frac{\epsilon}{2}
\end{align}
for a given tolerance $\epsilon$.
Then, $\int_l^r \Fde(x) \dd{x}$ is computed by using the $m$-point trapezoidal rule.
The second algorithm adaptively computes $\int_l^r \Fde(x) \dd{x}$ until the discretization error is smaller than or equal to $\epsilon/2$.
Hence, the total error of the second algorithm is smaller than or equal to $\epsilon$.

The analysis of the truncation error for a given finite interval is shown in Section \ref{sec:2-1:est-terr}.
In Section \ref{sec:2-2:set-int}, we propose a truncation method for \eqref{eq:de_err_goal}.
We explain the algorithms in Section \ref{sec:2-3:algorithm}.

\subsection{Truncation error analysis}
\label{sec:2-1:est-terr}
The truncation error is bound as follows:
\begin{lemma}\label{thm:abserr_lr}
  Let $A$ have no eigenvalues on the closed negative real axis and $\alpha\in(0,1)$.
  For a given interval $[l,r]$, let $a=\exp(\alpha\pi\sinh(l)/2)$ and $b=\exp(\alpha\pi\sinh(r)/2)$.
  For $a\le (2\|A^{-1}\|)^{-\alpha}$ and $b\ge (2\|A\|)^\alpha$, we have
  \begin{align}\label{eq:abserr}
    & \norm{
      \int_{-\infty}^{\infty} \Fde(x) \dd{x} - \int_l^r \Fde(x) \dd{x}
      }\\
    & \quad \le \frac{\sin(\alpha\pi)[\alpha+(1+\alpha)\|I\|]}{\alpha\pi(1+\alpha)}a
      + \frac{\sin(\alpha\pi)(3-2\alpha)\|A\|}{\pi(1-\alpha)(2-\alpha)} b^{1-1/\alpha}.
  \end{align}
\end{lemma}

\begin{proof}
  From the triangle inequality, the truncation error is bounded by
  \begin{align}
    \label{eq:de_truncationerr}
    &\left\lVert
      \int_{-\infty}^{\infty} \Fde(x) \dd{x}
      - \int_l^r \Fde(x) \dd{x}
    \right\rVert\\
    &\quad \le
    \left\lVert \int_{-\infty}^l \Fde(x) \dd{x} \right\rVert
    + \left\lVert \int_r^{\infty} \Fde(x) \dd{x} \right\rVert.
  \end{align}
  We prove Lemma \ref{thm:abserr_lr} by showing
  \begin{align}\label{eq:err_za}
    \norm{
      \int_{-\infty}^{l} \Fde(x) \dd{x}
    }
    = \norm{
      \int_0^a F(t) \dd{t}
    }
    \le \frac{\sin(\alpha\pi)[\alpha+(1+\alpha)\|I\|]}{\alpha\pi(1+\alpha)}a
  \end{align}
  and
  \begin{align}\label{eq:err_bi}
    \norm{
      \int_r^\infty \Fde(x) \dd{x}
    }
    = \norm{
      \int_b^\infty F(t) \dd{t}
    }
    \le \frac{\sin(\alpha\pi)(3-2\alpha)\|A\|}{\pi(1-\alpha)(2-\alpha)} b^{1-1/\alpha}.
  \end{align}

  First, we show \eqref{eq:err_za}.
  For all $t \in [0,a]$ where $a\le(2\|A^{-1}\|)^{-\alpha}$, $t\le (2\|A^{-1}\|)^{-\alpha}$, and therefore, $\|t^{1/\alpha}A^{-1}\| \le 1/2~(<1)$.
  By applying the Neumann series expansion to $(t^{1/\alpha}I+A)^{-1}$, we get the following:
  \begin{align}
    (t^{1/\alpha}I+A)^{-1}
    = A^{-1}[I-(-t^{1/\alpha}A^{-1})]^{-1}
    = A^{-1} \sum_{k=0}^\infty (-1)^k A^{-k} t^{k/\alpha}.
  \end{align}
  Therefore, the integral can be rewritten as follows:
  \begin{align}\label{eq:use_WMtest}
    & A\int_0^a (t^{1/\alpha}I+A)^{-1} \dd{t}
      = A\int_0^a A^{-1} \left[\sum_{k=0}^\infty (-1)^k A^{-k} t^{k/\alpha}\right]\dd{t}\\
    & \quad = \sum_{k=0}^\infty (-1)^kA^{-k}\left[\int_0^a t^{k/\alpha} \dd{t}\right]
      = \sum_{k=0}^\infty (-1)^kA^{-k}\left(\frac{1}{k/\alpha+1}a^{k/\alpha+1}\right)\\
    & \quad = aI + \sum_{k=1}^\infty (-1)^kA^{-k}\left(\frac{1}{k/\alpha+1}a^{k/\alpha+1}\right).
  \end{align}
  Thus, we get the following:
  \begin{align}\label{eq:tmp01}
    & \norm{\int_0^a F(t) \dd{t}}
      = \norm{\frac{\sin(\alpha\pi)}{\alpha\pi}A\int_0^a (t^{1/\alpha}I+A)^{-1} \dd{t}}\\
    & \quad = \frac{\sin(\alpha\pi)}{\alpha\pi} \norm{aI + \sum_{k=1}^\infty (-1)^kA^{-k}\left(\frac{1}{k/\alpha+1}a^{k/\alpha+1}\right)}\\
    & \quad \le \frac{\sin(\alpha\pi)}{\alpha\pi} \left[
      a\norm{I}
      + \sum_{k=1}^\infty \norm{A^{-1}}^{k}\left(\frac{1}{k/\alpha+1}a^{k/\alpha+1}\right)
    \right].
  \end{align}
  Furthermore, from $a\le (2\|A^{-1}\|)^{-\alpha}$, we have the following: 
  \begin{align}
    & \sum_{k=1}^\infty \norm{A^{-1}}^{k}\left(\frac{1}{k/\alpha+1}a^{k/\alpha+1}\right)
      = \sum_{k=1}^\infty \frac{\alpha a}{k+\alpha} \norm{a^{1/\alpha} A^{-1}}^k\\
    & \quad \le \sum_{k=1}^\infty \frac{\alpha a}{1+\alpha} \left(\frac{1}{2}\right)^k
      =\frac{\alpha}{1+\alpha}a. \label{eq:tmp02}
  \end{align}
  We get \eqref{eq:err_za} by substituting \eqref{eq:tmp02} in \eqref{eq:tmp01}.

  Next, we show \eqref{eq:err_bi}.
  The outline of this proof is almost the same as that of \eqref{eq:err_za}.
  For all $t\in[b,\infty)$ with $b\ge (2\|A\|)^\alpha$, $\|t^{-1/\alpha}A\| \le 1/2 ~ (<1)$.
  By applying the Neumann series expansion to $(t^{1/\alpha}I+A)^{-1}$, we get
  \begin{align}
    (t^{1/\alpha}I+A)^{-1}
    = t^{-1/\alpha}[I-(-t^{-1/\alpha}A)]^{-1}
    = t^{-1/\alpha}\sum_{k=0}^\infty (-1)^kt^{-k/\alpha} A^k.
  \end{align}
  Therefore, the integral can be rewritten as follows:
  \begin{align}
    & \int_b^\infty(t^{1/\alpha}I+A)^{-1}\dd{t}
      =\int_b^\infty\left[t^{-1/\alpha}\sum_{k=0}^\infty (-1)^kt^{-k/\alpha} A^k\right]\dd{t}\\
    &\quad =\sum_{k=0}^\infty\left[(-1)^kA^k\int_b^\infty t^{-(k+1)/\alpha} \dd{t}\right]
     =\sum_{k=0}^\infty(-1)^kA^k \frac{\alpha}{k+1-\alpha}b^{1-(k+1)/\alpha}\\
    &\quad =\frac{\alpha}{1-\alpha}b^{1-1/\alpha}I+\sum_{k=1}^\infty(-1)^kA^k \frac{\alpha}{k+1-\alpha}b^{1-(k+1)/\alpha}. \label{eq:tmp07}
  \end{align}
  Thus,
  \begin{align} \label{eq:tmp03}
    & \norm{\int_b^\infty F(t) \dd{t}} = \norm{\frac{\sin(\alpha\pi)}{\alpha\pi}A\int_b^\infty (t^{1/\alpha} I+A)^{-1} \dd{t}}\\
    &\quad = \norm{\frac{\sin(\alpha\pi)}{\alpha\pi}\left[
      \frac{\alpha}{1-\alpha}b^{1-1/\alpha}A
      + A\sum_{k=1}^\infty(-1)^kA^k \frac{\alpha}{k+1-\alpha}b^{1-(k+1)/\alpha}
      \right]} \\
    &\quad  \le\frac{\sin(\alpha\pi)}{\pi}\|A\|\left[
      \frac{1}{1-\alpha}b^{1-1/\alpha}
      + \sum_{k=1}^\infty \|A\|^k \frac{1}{k+1-\alpha}b^{1-(k+1)/\alpha}
      \right].
  \end{align}
  Moreover, from $b \ge (2\|A\|)^\alpha$, we get the following:
  \begin{align} \label{eq:tmp04}
    &\sum_{k=1}^\infty \|A\|^k \frac{1}{k+1-\alpha}b^{1-(k+1)/\alpha}
      = b^{1-1/\alpha}\sum_{k=1}^\infty \frac{1}{k+1-\alpha} \|b^{-1/\alpha} A\|^k\\
    &\quad \le b^{1-1/\alpha} \sum_{k=1}^\infty \frac{1}{2-\alpha} \left(\frac{1}{2}\right)^k
    =\frac{1}{2-\alpha}b^{1-1/\alpha}.
  \end{align}
  We obtain \eqref{eq:err_bi} by substituting \eqref{eq:tmp04} in \eqref{eq:tmp03}.

  In conclusion, by combining \eqref{eq:de_truncationerr}, \eqref{eq:err_za} and \eqref{eq:err_bi}, we get \eqref{eq:abserr}.
\end{proof}

\begin{remark}
The technique of applying the Neumann series expansion to the integrand was considered in \cite[Thm.\ 2.1]{cardoso_computation_2012}, and an upper bound was derived as follows:
  \begin{align}
    \norm{\int_b^\infty (t^{1/\alpha}I+A)^{-1} \dd{t}} \le \frac{2b^{1-1/\alpha}}{1/\alpha - 1}.
  \end{align}
  Hence, the following relation holds:
  \begin{align}\label{eq:err_cardoso}
    \norm{\int_b^\infty F(t) \dd{t}} \le \frac{2\sin(\alpha\pi)\|A\|}{\pi(1-\alpha)} b^{1-1/\alpha}.
  \end{align}
  Note that the upper bound \eqref{eq:err_bi} is smaller than \eqref{eq:err_cardoso} because
  \begin{align}
    &\frac{2\sin(\alpha\pi)\|A\|}{\pi(1-\alpha)} b^{1-1/\alpha}
    - \frac{\sin(\alpha\pi)(3-2\alpha)\|A\|}{\pi(1-\alpha)(2-\alpha)} b^{1-1/\alpha}\\
    &\quad = \frac{\sin(\alpha\pi)\|A\|}{\pi(1-\alpha)} b^{1-1/\alpha} \left(2-\frac{3-2\alpha}{2-\alpha}\right)
    = \frac{\sin(\alpha\pi)\|A\|}{\pi(1-\alpha)} b^{1-1/\alpha} \frac{1}{2-\alpha}
    >0.
  \end{align}
  This difference originates from the separation of the sum into two terms in \eqref{eq:tmp07}.
\end{remark}

\subsection{Truncation method based on the error analysis}\label{sec:2-2:set-int}
On the basis of Lemma \ref{thm:abserr_lr}, we show a method to truncate the infinite interval into a finite interval within the given tolerance below.
\begin{proposition}\label{thm:setinterval}
For given $\epsilon>0$, let
\begin{align}
  l = \asinh\left(\frac{2\log(a)}{\alpha\pi}\right),\qquad r = \asinh\left(\frac{2\log(b)}{\alpha\pi}\right),
\end{align}
where
\begin{align}
  &a= \min\left\{
    \frac{\epsilon}{4}
    \frac{
      \pi\alpha(1+\alpha)
    }{
      \sin(\alpha\pi)(1+2\alpha)
    },(2\|A^{-1}\|_2)^{-\alpha}
  \right\},\\
  &b = \max\left\{
    \left[
      \frac{\epsilon}{4}
      \frac{
        \pi(1-\alpha)(2-\alpha)
      }{
        \sin(\alpha\pi)(3-2\alpha)\norm{A}_2
      }
  \right]^{\alpha/(\alpha-1)},
  (2\|A\|_2)^\alpha
  \right\}.
\end{align}
Then, the following holds: 
\begin{align}\label{eq:relerr}
  \norm{\int_{-\infty}^\infty \Fde(x) \dd{x} - \int_l^r \Fde(x) \dd{x}}_2 \le \frac{\epsilon}{2}.
\end{align}
\end{proposition}

\begin{proof}
From the definitions of $a$ and $b$, $a\le (2\|A^{-1}\|_2)^{-\alpha}$ and $b \ge (2\|A\|_2)^\alpha$ are true.
Therefore, the interval $[l,r]$ satisfies the assumptions for Lemma \ref{thm:abserr_lr}.
Hence, it follows that
\begin{align}\label{eq:tmp05}
  & \norm{\int_{-\infty}^\infty \Fde(x) \dd{x} - \int_l^r \Fde(x)\dd{x}}_2\\
  & \quad \le \frac{\sin(\alpha\pi)(1+2\alpha)}{\alpha\pi(1+\alpha)}a
  + \frac{\sin(\alpha\pi)(3-2\alpha)\|A\|_2}{\pi(1-\alpha)(2-\alpha)} b^{1-1/\alpha}.
\end{align}
From the definition of $a$ and $b$, we have
\begin{align}\label{eq:tmp06}
  \frac{\sin(\alpha\pi)(1+2\alpha)}{\alpha\pi(1+\alpha)}a \le \frac{\epsilon}{4},
  \qquad
  \frac{\sin(\alpha\pi)(3-2\alpha)\|A\|_2}{\pi(1-\alpha)(2-\alpha)} b^{1-1/\alpha}\le \frac{\epsilon}{4}.
\end{align}
We obtain \eqref{eq:relerr} by substituting \eqref{eq:tmp06} in \eqref{eq:tmp05}.
\end{proof}

\subsection{Algorithms}\label{sec:2-3:algorithm}
We first explain the algorithm for computing $A^\alpha$ on the basis of the $m$-point DE formula.
This algorithm computes $[l,r]$ that satisfies \eqref{eq:relerr} on the basis of Proposition \ref{thm:setinterval} and computes $\int_l^r \Fde(x) \dd{x}$ by the $m$-point trapezoidal rule.
The details of the algorithm are given in Algorithm \ref{alg:1}.

\begin{algorithm}
  \caption{$m$-point DE formula for computing $A^\alpha$}
  \begin{algorithmic}[1]
    \Statex \textbf{Input} $A$, $\alpha\in(0,1)$, $\epsilon>0$, $m$.
    \State $l, r = \texttt{GetInterval}(A,\alpha,\epsilon)$
    \State Set $\tildeFde(x) = \exp(\alpha\pi\sinh(x)/2)\cosh(x)\left[\exp(\pi\sinh(x)/2)I+A \right]^{-1}$
    \State $h = (r - l) / (m - 1)$
    \State $T = h[\tildeFde(l) + \tildeFde(r)]/2 + h\sum_{k=1}^{m-2}\tildeFde(l+kh)$
    \Statex \textbf{Output} $\sin(\alpha\pi) AT/2 \approx A^\alpha$
    \Statex
    \Function{GetInterval}{$A$, $\alpha$, $\epsilon$}
      \State Compute $\norm{A}_2$, $\|A^{-1}\|_2$
      \State $a_1 = [\alpha\pi(1+\alpha)\epsilon]/[4\sin(\alpha\pi)(1+2\alpha)], \quad a_2 = (2\|A^{-1}\|_2)^{-\alpha}$
      \State $a = \min\{a_1, a_2\}$
      \State $b_1 =
        [
          \pi(1-\alpha)(2-\alpha) \epsilon
        ]^{\alpha/(\alpha-1)}
        /
        [
          4\sin(\alpha\pi)(3-2\alpha)\norm{A}_2
        ]^{\alpha/(\alpha-1)}
        ,
        \quad
        b_2 = (2\|A\|_2)^\alpha
      $
      \State $b = \max \{b_1, b_2\}$
      \State $l = \asinh(2\log(a)/\alpha\pi), \quad r = \asinh(2\log(b)/\alpha\pi)$
      \State \Return $l,r$
    \EndFunction
  \end{algorithmic}
  \label{alg:1}
\end{algorithm}

When the tolerance $\epsilon$ is sufficiently small, the accurate computation of $\|A\|_2$, $\|A^{-1}\|_2$ in Step 6 of Algorithm \ref{alg:1} is not necessary, because the errors originating from these values have little effect on the truncation error.
Assume that $\epsilon$ is sufficiently small, and $[a,b] = [a_1,b_1]$.
Let $\Delta\in\bbR$ be such that the computed 2-norm of $A$ is equal to $\|A\|_2(1+\Delta)$.
Then, the computed value of $b$ without rounding errors is
\begin{align}
  b = \left[
    \frac{
      \pi(1-\alpha)(2-\alpha)
    }{
      4\sin(\alpha\pi)(3-2\alpha)\norm{A}_2
    } \epsilon_b
  \right]^{\alpha/(\alpha-1)},
\end{align}
where $\epsilon_b = \epsilon/(1 + \Delta)$.
Hence, the upper bound of the truncation error is
\begin{align}
  \left\lVert
      \int_{-\infty}^{\infty} \Fde(x) \dd{x}
      - \int_{l}^{r} \Fde(x) \dd{x}
    \right\rVert
  \le \frac{1}{2}(\epsilon + \epsilon_b) \le \frac{1}{2}\epsilon \left(1 + \frac{1}{1-|\Delta|}\right).
\end{align}
For example, when $\Delta = 10^{-2}$, the upper bound of the truncation error is $1.005\epsilon$ at most.
Therefore, the effect of these errors will be negligible.
In practice, the effect of the error can be canceled by setting $\epsilon$ of Algorithm 1 to $\epsilon = \tilde{\epsilon} / (1 + 1/(1 - |\tilde{\Delta}|))$, where $\tilde {\epsilon}$ is the tolerance for the truncation error of the DE formula and $\tilde{\Delta}$ is the tolerance for the relative error used in the computation of $\|A\|_2$.

It may be desirable to apply bounds for the relative truncation error rather than for the absolute truncation error; in other words, one may wish to select $[l,r]$ that satisfies 
\begin{align}
  \frac{\norm{A^\alpha - \int_l^r \Fde(x)\dd{x}}_2}{\|A^\alpha\|_2} \le \frac{\tilde{\epsilon}}{2}
\end{align}
for a given tolerance $\tilde{\epsilon}$.
In this case, one can select such an interval by setting $\epsilon = \rho(A)^\alpha \tilde{\epsilon}$ for the input of Algorithm \ref{alg:1}.
This is because $\rho(A)^\alpha = \rho(A^\alpha)$ bounds $\|A^\alpha\|_2$ from below.
The relation $\rho(A)^\alpha = \rho(A^\alpha)$ can be verified from the fact that the eigenvalues of $A^\alpha$ are $\lambda_i^\alpha$, where $\lambda_i$ are the eigenvalues of $A$, see, e.g., \cite[Thm.\ 1.13]{higham_functions_2008}.

Before explaining the adaptive quadrature algorithm, we consider an a posteriori estimation of the discretization error.
Let $m$ be the number of abscissas, $h=(r-l)/(m-1)$ is the mesh size, and $T(h)$ is the trapezoidal rule of the mesh size $h$:
\begin{align}
T(h) := \frac{h}{2} \left(F_{\mathrm{DE}}(l) + F_{\mathrm{DE}}(r)\right) + h \sum_{k=1}^{m-2} F_{\mathrm{DE}}(l+kh).
\end{align}
Then, $T(h/2)$ can be computed as follows:
\begin{align}
T\left(\frac{h}{2}\right) = \frac{1}{2}T(h) + \frac{h}{2}\sum_{k=1}^{m-1} F_{\mathrm{DE}}\left(l+(2k-1)\frac{h}{2}\right).
\end{align}
The upper bound on the discretization error of $T(h/2)$ can be obtained by applying the discussion in \cite[p.\ 424]{cardoso_computation_2012} to the trapezoidal rule.
Assume that $h$ is small enough to satisfy $\|T(h) - \int_l^r\Fde(x)\,\mathrm{d}x\|\le \mu$ and $\|T(h/2) - \int_l^r\Fde(x)\,\mathrm{d}x\|\le c\mu$, with $0 < c \le 1/2$.
If $\|T(h/2) - T(h)\| \le \tilde{\mu}$, then
\begin{align}
  \left\lVert T(h) - \int_l^r F_{\mathrm{DE}}(x)\,\mathrm{d}x \right\rVert
  \le \left\lVert T(h) - T\left(\frac{h}{2}\right) \right\rVert
  + \left\lVert T\left(\frac{h}{2}\right) - \int_l^r F_{\mathrm{DE}}(x)\,\mathrm{d}x \right\rVert,
\end{align}
that is, $\mu \le \tilde{\mu} + c\mu$, or equivalently, $\mu \le \tilde{\mu}/(1-c)$.
Hence,
\begin{align}\label{eq:trap_err_estimation}
  \left\lVert \int_l^r F_{\mathrm{DE}}(x)\,\mathrm{d}x - T\left(\frac{h}{2}\right) \right\rVert
  \le \left\lVert T(h) - T\left(\frac{h}{2}\right) \right\rVert.
\end{align}

Now we shall state the adaptive quadrature algorithm.
First, the algorithm computes $[l,r]$, as shown in Algorithm \ref{alg:1}.
Then, the algorithm repeats the computation of $\int_l^r \Fde(x) \dd{x}$ by increasing the number of abscissas until the discretization error is smaller than $\epsilon/2$.
The details of the algorithm are shown in Algorithm \ref{alg:2}.
\begin{algorithm}
  \caption{Adaptive quadrature algorithm for $A^\alpha$ based on the DE formula}
  \begin{algorithmic}[1]
    \Statex \textbf{Input} $A$, $\alpha\in(0,1)$, $\epsilon>0$, $m_0\ge 2$.
    \State $l, r = \texttt{GetInterval}(A,\alpha,\epsilon)$
    \Comment{\texttt{GetInterval} is defined in Algorithm \ref{alg:1}.}
    \State Set $\tildeFde(x) = \exp(\alpha\pi\sinh(x)/2)\cosh(x)\left[\exp(\pi\sinh(x)/2)I+A \right]^{-1}$
    \State $s = -1$
    \State $h_0 = (r-l) / (m_0 - 1)$
    \State $T_0 = h_0[\tildeFde(l) + \tildeFde(r)]/2 + h_0\sum_{k=1}^{m-2}\tildeFde(l+kh)$
    \Repeat
      \State $s = s+1$
      \State $h_{s+1} = h_s / 2$
      \State $T_{s+1} = T_s/2 + h_{s+1}\sum_{k=1}^{m_s-1} \tildeFde(l+(2k-1)h_{s+1})$
      \State $m_{s+1} = 2m_s - 1$
      
    \Until{$\sin(\alpha\pi)\|AT_{s+1} - AT_s\|/2 > \epsilon/2$}
    \State $T = T_{s+1}$
    \Statex \textbf{Output} $\sin(\alpha\pi) AT/2 \approx A^\alpha$
  \end{algorithmic}
  \label{alg:2}
\end{algorithm}

\begin{remark}
  The computational cost of Algorithm \ref{alg:1} for a dense matrix $A$ is about $(8m/3 + 2) n^3$ if the computational cost for computing norms $\|A\|_2$ and $\|A^{-1}\|_2$ is $\calO(n^2)$, and that of Algorithm \ref{alg:2} is about $(8m_{s+1}/3 + 2(s+1)) n^3$.
  Further, the cost incurred in computing $A^\alpha\bmb$ with Algorithm \ref{alg:1} is $mc_{\mathrm{abscissa}} + c_{\mathrm{mul}} + c_{\mathrm{norm}}$,
  where $c_{\mathrm{abscissa}}$ is the computational cost of computing $\tilde{F}_{\mathrm{DE}}(x)\bmb$,
  $c_{\mathrm{mul}}$ is that of a matrix-vector multiplication,
  and $c_{\mathrm{norm}}$ is that of computing norms $\|A\|_2$ and $\|A^{-1}\|_2$.
  The computational cost of Algorithm \ref{alg:2} is $m_{s+1}c_{\mathrm{abscissa}} + (s+1)c_{\mathrm{mul}}+ c_{\mathrm{norm}}$.
  Because one can use low-accuracy norms, the total computational cost will be almost proportional to the total number of abscissas.
\end{remark}

\section{Convergence of the DE formula for HPD matrices}\label{sec:3:convergence}
When $A$ is an HPD matrix, the analysis of a quadrature formula for $A^\alpha$ can be reduced to that of the quadrature formula for the scalar fractional power.
For the DE formula, the following relation holds: 
\begin{align}\label{eq:3-01}
  &\norm{\int_{-\infty}^\infty \Fde(x) \dd{x}-h\sum_{k=-\infty}^\infty \Fde(kh)}_2\\
  &\qquad = \max_{\lambda\in\Lambda(A)}\left|\int_{-\infty}^\infty \fde(x,\lambda) \dd{x}-h\sum_{k=-\infty}^\infty \fde(kh,\lambda)\right|,
\end{align}
where $h$ is mesh size, $\Lambda(A)$ is the spectrum of $A$, and $\fde$ is the scalar counterpart of $\Fde(x)$. Further, $\fde$ can be expressed as follows:
\begin{align}\label{eq:fde}
  \fde(x,\lambda)=\frac{\sin(\alpha\pi)\lambda}{2}  \frac{\exp(\alpha\pi\sinh(x)/2)\cosh(x)}{\exp(\pi\sinh(x)/2)+\lambda}.
\end{align}
Section \ref{subs:3-1_fde} presents our analysis of the discretization error of the trapezoidal rule for $\int_{-\infty}^\infty \fde \dd{x}$, i.e., the DE formula for $\lambda^\alpha$.
The error of the DE formula for $A^\alpha$ is discussed in Section \ref{subs:3-2_Fde}, and the DE formula is compared with Gaussian quadrature in Section \ref{subs:3-3_comparison}.

\subsection{Discretization error of the DE formula for \texorpdfstring{$\lambda^\alpha$}{the scalar fractional power}}\label{subs:3-1_fde}
First, we recall the following upper bound on the discretization error of the trapezoidal rule for an integral of an analytic function on the real axis:
\begin{theorem}[see, {\cite[Thm. 5.1]{trefethen_exponentially_2014}}]\label{thm:convergence_de}
  Suppose $g$ is analytic in the strip $\calD_d = \{z\in\bbC\colon |\rmIm(z)| < d\}$ for some $d>0$.
  Further, suppose that $g(z)\to 0$ uniformly as $|z|\to \infty$ in the strip, and for some $c$, it satisfies
  \begin{align}\label{eq:thm3.1condition}
    \int_{-\infty}^\infty |g(x+\rmi y)|\dd{x} \le c
  \end{align}
  for all $y\in(-d,d)$.
  Then, for any $h>0$, the sum $h\sum_{k=-\infty}^\infty g(kh)$ exists and satisfies
  \begin{align}
    \left|h\sum_{k=-\infty}^\infty g(kh) - \int_{-\infty}^\infty g(x)\dd{x}\right|
    \le \frac{2c}{\exp(2\pi d/h) - 1}.
  \end{align}
\end{theorem}

To apply Theorem \ref{thm:convergence_de} to $\fde$, we need to identify the strip in which $\fde$ is analytic.
We identify such a region using the following proposition:
\begin{proposition}\label{thm:convergence_fde}
  For $\lambda>0$, $\fde(z,\lambda)$ defined in \eqref{eq:fde} is analytic in the strip $\calD_{d_0(\lambda)}$, where
  \begin{align}\label{eq:d0}
    d_0(\lambda) = \arcsin\left(
      \sqrt{\frac{(\log \lambda)^2+ 5\pi^2/4- \sqrt{[(\log \lambda)^2+ 5\pi^2/4]^2-\pi^4}}{\pi^2/2}}
    \right).
  \end{align}
\end{proposition}

\begin{proof}
  We prove Proposition \ref{thm:convergence_fde} by calculating the imaginary part of a pole of $\fde$ that is the closest to the real axis among the poles of $\fde$.
  Because $\exp(\alpha\pi\sinh(z)/2)\cosh(z)$ is analytic in $\bbC$, we consider only the poles of $1/[\exp(\pi\sinh(z))/2)+\lambda]$.

  In fact, the poles are calculated exactly.
  Consider the following equation:
  \begin{align}\label{eq:3-06}
    \exp\left(\frac{2}{\pi}\sinh z\right) = -\lambda.
  \end{align}
  Let $x$ and $y$ be the real and imaginary parts of $z$ in \eqref{eq:3-06}, respectively.
  Then, \eqref{eq:3-06} can be rewritten as
  \begin{align}
    & \exp\left(\frac{\pi}{2}\sinh z\right)=\exp\left(\frac{\pi}{2}\sinh x\cos y+ \rmi\frac{\pi}{2}\cosh x \sin y \right), \label{eq:3-07}\\
    & -\lambda=\exp(\log\lambda +\rmi k\pi)\quad(k=\pm 1,\pm 3,\pm 5\dots). \label{eq:3-08}
  \end{align}
  From the right-hand side of \eqref{eq:3-07} and that of \eqref{eq:3-08}, it follows that
  \begin{align}\label{eq:3-09}
    \begin{cases}
      \pi\sinh(x)\cos(y)/2 = \log \lambda,\\
      \pi\cosh(x)\sin(y)/2 = k\pi.
    \end{cases}
  \end{align}
  To solve \eqref{eq:3-09}, we consider two cases: $\lambda$ is 1 or not.
  In case $\lambda = 1$, equation \eqref{eq:3-09} leads $\cos y = 0$.
  Therefore, the minimum absolute solution $y$ is $\pm\pi/2$.

  In case $\lambda\ne 1$, $\cos y \ne 0$ because $\log \lambda \ne 0$.
  After squaring the left and right of both equations of \eqref{eq:3-09}, we obtain
  \begin{align}
    \left(\frac{\pi^2}{4} + \frac{(\log\lambda)^2}{1-\sin^2y}\right) \sin^2 y = k^2\pi^2
    \label{eq:3-11}
  \end{align}
  by substituting the relations $\cosh^2 x  = 1 + \sinh^2 x$ and $\cos^2 y=1-\sin^2 y (\ne 0)$.
  Equation \eqref{eq:3-11} can be solved easily because \eqref{eq:3-11} is a quadratic equation for $\sin^2 y$.
  Thus, the imaginary part of the solutions are as follows:
  \begin{align} \label{eq:3-12}
    y_k = \arcsin\left(\sqrt{\frac{
      (\log \lambda)^2 + \pi^2/4 + k^2\pi^2 - \sqrt{[(\log \lambda)^2 + \pi^2/4 + k^2\pi^2]^2 - k^2\pi^4}
    }{\pi^2/2}}\right),
  \end{align}
  We can check that $|y_1| = |y_{-1}| \le |y_k|$ for all $|k|>1$.
  Therefore, the imaginary part of the closest poles are $\pm y_1$.

  Because $y_1 = \pi/2$ when $\lambda=1$, we can combine the two cases about $\lambda$.
  In conclusion, $\fde(z,\lambda)$ is analytic in $\calD_{d_0(\lambda)}$, where $d_0(\lambda) = y_1$.
\end{proof}

To apply Theorem \ref{thm:convergence_de} to $\fde$, we make sure that $\fde$ satisfies the condition \eqref{eq:thm3.1condition}.
Let $\delta$ be a sufficiently small positive number.
Then, $\fde$ is continuous and bounded in $\calD_{d_0(\lambda) - \delta}$.
In addition, by using the results in Appendix \ref{sec:app-calc}, we can show that there exist constants $c_l,c_r>0$ and $l, r\in\bbR$ such that for all $|y| \le d_0(\lambda) - \delta\; (< \pi/2)$,
\begin{align}
  &|\fde(x+\rmi y)| < c_l \exp \left(-\frac{\alpha\pi \cos(d_0(\lambda)-\delta)}{2} \sinh(|x|)\right) \cosh x \qquad (x < l),
  \label{eq:condition31l}\\
  &|\fde(x+\rmi y)| < c_r \exp\left(-\frac{(1-\alpha)\pi\cos(d_0(\lambda) - \delta)}{2} \sinh x\right) \cosh x
  \qquad (x > r).
  \label{eq:condition31r}
\end{align}
Therefore, $\fde$ satisfies condition \eqref{eq:thm3.1condition}.

To conclude, the discretization error of the DE formula for $\lambda^\alpha$ is estimated as
\begin{align}\label{eq:convrate_scde}
  \left|\lambda^\alpha - h\sum_{k=-\infty}^\infty \fde(kh,\lambda)\right|
  \le \calO \left(\exp\left(-\frac{2\pi d_0(\lambda)}{h}\right)\right).
\end{align}
The estimation \eqref{eq:convrate_scde} has a limitation in the sense that it only considers the convergence rate.
The error of the DE formula goes to 0 as $\lambda \to 0$ because the $\fde(x,\lambda) \to 0$.
However, \eqref{eq:convrate_scde} does not indicate it because $d_0(\lambda) \to 0$ as $\lambda \to 0$.

\subsection{Convergence of the DE formula for \texorpdfstring{$A^\alpha$}{the matrix fractional power}}\label{subs:3-2_Fde}
First, let us consider the eigenvalues of $A$ that determine the convergence rate of the DE formula for $A^\alpha$.
Since the error of the DE formula for $\lambda^\alpha$ is $\calO(\exp(-2\pi d_0(\lambda)/h))$, its convergence is slow when $|\log(\lambda)|$ is large, i.e., $\lambda$ is far away from $1$.
This is because $d_0(\lambda)$ decreases monotonically when $\lambda > 1$ and increases monotonically when $0<\lambda<1$.
In addition, the convergence rate of the DE formula for $\lambda^\alpha$ is approximately equal to that of the DE formula for $(1/\lambda)^\alpha$ because $d_0(\lambda)=d_0(1/\lambda)$.
This relation $d_0(\lambda)=d_0(1/\lambda)$ is true because $\lambda$ appears only in the form of $(\log \lambda)^2$ in $d_0(\lambda)$.
These properties of $d_0(\lambda)$ can be checked in Figure \ref{fig:1}, which shows $d_0(\lambda)$ for $\lambda\in [10^{-16}, 10^{16}]$.
\begin{figure}
  \centering
  \includegraphics{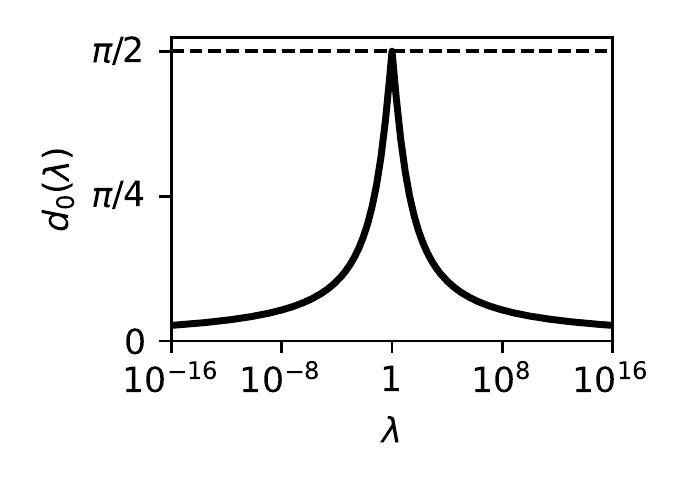}
  \caption{Distance $d_0(\lambda)$ (in \eqref{eq:d0}) of a strip where $\fde(z,\lambda)$ is analytic. Note that the horizontal axis is logarithmic.}
  \label{fig:1}
\end{figure}
Therefore, for sufficiently small $h$, the discretization error depends on the maximum and the minimum eigenvalues of $A$, say $\lambdamax,\lambdamin$, and the error can be written as
\begin{align}\label{eq:tmp08}
  \norm{A^\alpha - h\sum_{k=-\infty}^\infty \Fde(kh)}_2
  = \max_{\lambda\in\{\lambdamax,\lambdamin\}}\left|\lambda^\alpha-h\sum_{k=-\infty}^\infty \fde(kh,\lambda)\right|.
\end{align}
For $h$ that is not sufficiently small, the discretization error may depend on eigenvalues other than the extreme eigenvalues.
This is because the error of the DE formula depends on the coefficient multiplying the exponential factor in the discretization error as well as the exponential factor.

Now, we assume that $\lambdamax\lambdamin=1$ because this assumption minimizes the value
\begin{align}
  \max \{ |\log(\lambdamax)|, |\log(\lambdamin)| \},
\end{align}
which determines the convergence rate of the DE formula.
Note that the condition $\lambdamax\lambdamin=1$ can be satisfied without loss of generality because we can compute $A^\alpha = (\lambdamax\lambdamin)^{\alpha/2} (A/\sqrt{\lambdamax\lambdamin})^\alpha$, where the product of the maximum and the minimum eigenvalues of $A/\sqrt{\lambdamax\lambdamin}$ is 1.
Under this assumption, we have
\begin{align} \label{eq:3-13}
  &\norm{A^\alpha - \sum_{k=-\infty}^\infty h\Fde(kh)}_2 \le \calO\left(\exp\left(-\frac{2\pi d_0(\lambdamax)}{h}\right)\right)\\
  &\quad = \calO\left(\exp\left(-\frac{2\pi d_0(\sqrt{\kappa(A)})}{h}\right)\right)
\end{align}
because $d_0(\lambdamax) = d_0(1/\lambdamax) = d_0(\lambdamin)$.

Finally, to compare the DE formula with the Gaussian quadrature, let us rewrite the error \eqref{eq:3-13} with the number of abscissas $m$.
Assume that we have a finite interval $[l,r]$ for which the truncation error is smaller than the discretization error. Then, the total error can be rewritten as
\begin{align}\label{eq:rate_DE}
  &\norm{A^\alpha- \frac{h}{2}\left(\Fde(l) + \Fde(r)\right)- h\sum_{k=1}^{m-2} \Fde(l+kh)}_2\\
  &\quad \le \calO\left(\exp\left(-\frac{2\pi d_0\bigl(\sqrt{\kappa(A)}\bigr)}{r-l}m\right)\right),
\end{align}
where $h = (r-l)/(m-1)$.
The error \eqref{eq:rate_DE} can be estimated by computing $\lambda_{\mathrm{max}}^\alpha$ using the $m$-point DE formula.

\subsection{Comparison of the convergence speed}\label{subs:3-3_comparison}
In this subsection, we compare the convergence speeds of three quadrature formulas.
The first one, denoted by \texttt{DE}, is the $m$-point DE formula (Algorithm \ref{alg:1}).
The second one, denoted by \texttt{GJ1}, represents the application of the Gauss--Jacobi (GJ) quadrature to
\begin{align}\label{eq:ir_cardoso}
  A^\alpha = \frac{2\sin(\alpha\pi)}{\alpha\pi} A \int_{-1}^1 (1-u)^{1/\alpha-2}\left[(1+u)^{1/\alpha}I+(1-u)^{1/\alpha}A\right]^{-1} \dd{u},
\end{align}
which is obtained by applying $t(u)=(1+u)/(1-u)$ to \eqref{eq:ir_realaxis}.
Integral representations similar to \eqref{eq:ir_cardoso} are considered in (the literature)\cite{cardoso_computation_2012,fasi_computing_2018}.
The third one, denoted by \texttt{GJ2}, represents the application of the GJ quadrature to
\begin{align}\label{eq:ir_fasi}
  A^\alpha = \frac{2\sin(\alpha\pi)}{\pi} A \int_{-1}^1 (1-v)^{\alpha-1}(1+v)^{-\alpha}\left[(1-v)I+(1+v)A\right]^{-1} \dd{v},
\end{align}
which is obtained by applying $t(v)=(1-v)^\alpha/(1+v)^\alpha$ to \eqref{eq:ir_realaxis}.
Integral representations similar to \eqref{eq:ir_fasi} are considered in \cite{aceto_rational_2017,fasi_computing_2018}.

In \cite{fasi_computing_2018}, under the assumption that $\lambdamax\lambdamin=1$, the error of the \texttt{GJ2} is estimated\footnote{
  Generally, the error of the $m$-point Gaussian quadrature is represented as $\calO(\tau^{-2m})$ for some constant $\tau$.
  In this paper, we use the representation $\calO(\exp(-2\log(\tau)m))$ because we want to compare the convergence of Gaussian quadrature with that of the DE formula.
} to be
\begin{align}\label{eq:appa:rc-gj2}
  \calO\left(\exp\left(-2\log\left(\frac{1+[\kappa(A)]^{1/4}}{|1 - [\kappa(A)]^{1/4}|}\right)m\right)\right).
\end{align}
The error of \texttt{GJ1} when $\alpha$ is a unit fraction is estimated \cite{fasi_computing_2018} as 
\begin{align}\label{eq:appa:rc-gj1}
  \calO\left(\exp\left(-2\log\left(
    \frac{1+[\kappa(A)]^{\alpha/2} + \sqrt{2[\kappa(A)]^{\alpha/2}(1-\cos(\alpha\pi))}}{\sqrt{1+[\kappa(A)]^{\alpha}+2[\kappa(A)]^{\alpha/2}\cos(\alpha\pi)}}
  \right)m\right)\right).
\end{align}
When $\alpha$ is a non-unit fraction, the integrand in \eqref{eq:ir_cardoso} is not analytic on $u=\pm 1$, and the GJ quadrature for \eqref{eq:ir_cardoso} may not converge exponentially.

Based on the above discussion, we can represent the error of the three quadrature formulas in the form $\calO(\exp(-\phi m))$, where $\phi$ is a constant depending on the quadrature formulas, $\kappa(A)$, and $\alpha$.
Figure \ref{fig:2} shows the values of $\phi$ for $\alpha=0.1,~0.2,\dots,0.9$ and $\kappa(A)\in [1, 10^{16}]$.
For \texttt{DE}, a finite interval $[l,r]$ is obtained by using the subroutine \texttt{GetInterval} in Algorithm \ref{alg:1}.
The parameter $\epsilon$ is set to $\epsilon=2^{-53} \kappa(A)^{\alpha/2}$ so that the relative error is smaller than $2^{-53}\approx 1.1\times 10^{-16}$.
For \texttt{GJ1}, the speed $\phi$ is only considered when $\alpha$ is a unit fraction because of the analyticity of the integrand.
Note that the vertical axes in Figure \ref{fig:2} show the values $\phi$. Hence, the convergence is fast for large $\phi$.

\begin{figure}
  \centering
  \includegraphics[width=\linewidth]{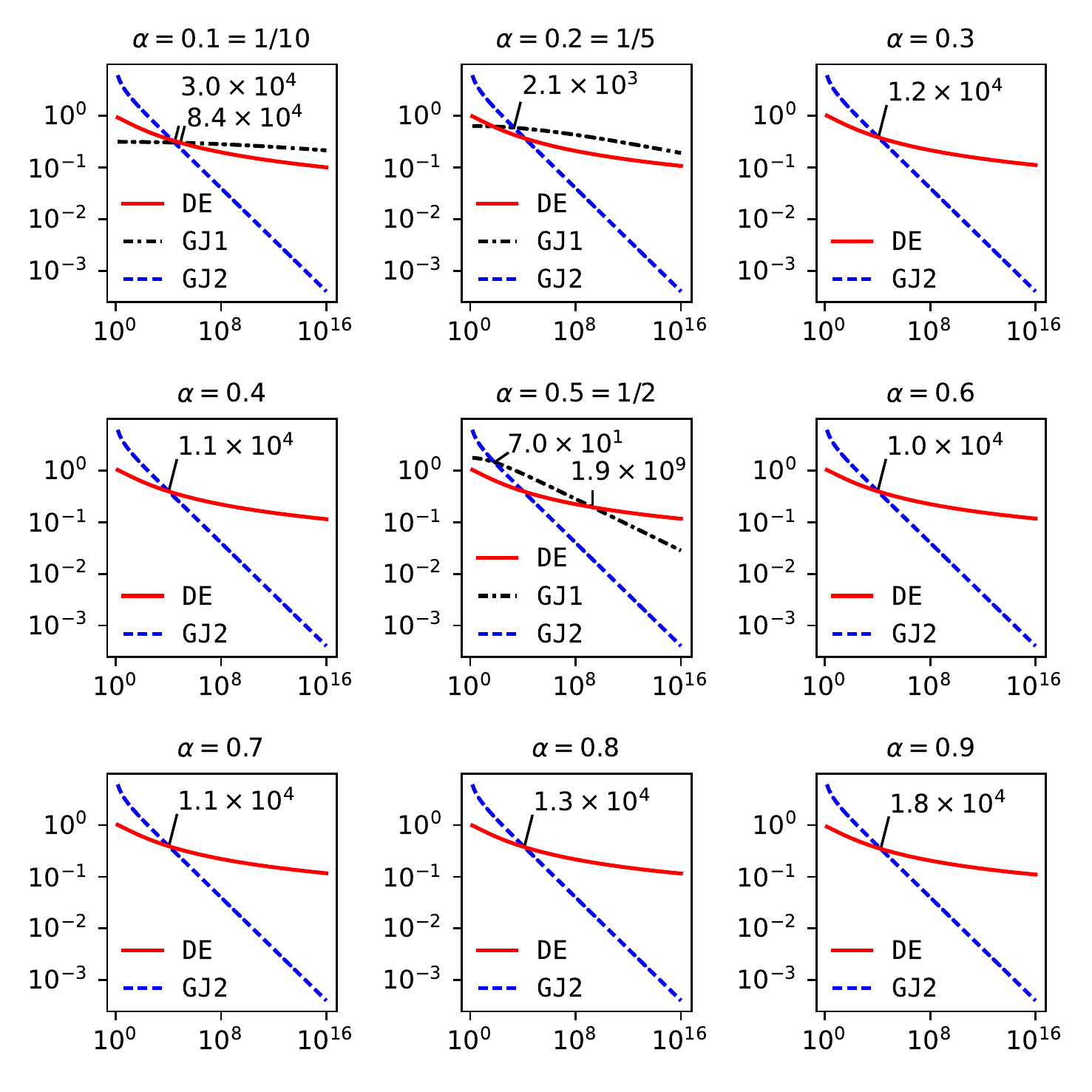}
  \caption{Convergence speed of the following quadrature formulas: the DE formula (\texttt{DE}), the GJ quadrature for \eqref{eq:ir_cardoso} (\texttt{GJ1}), and the GJ quadrature for \eqref{eq:ir_fasi} (\texttt{GJ2}).
  The upper bound of the error for \texttt{DE}, \texttt{GJ1}, and \texttt{GJ2} are shown in \eqref{eq:rate_DE}, \eqref{eq:appa:rc-gj1}, and \eqref{eq:appa:rc-gj2}, respectively.
  The horizontal axes are the condition number $\kappa(A)$.
  The vertical axes are the convergence speed, i.e., the constant $\phi$ in the error of the $m$-point quadrature formula $\exp(-\phi m)$.
  Therefore, the convergence of a quadrature formula is fast if the values are large.
  The condition numbers are indicated for the points where the fastest quadrature formula is changed.}
  \label{fig:2}
\end{figure}

Figure \ref{fig:2} shows that the convergence speed of \texttt{DE} is higher than that of \texttt{GJ2} when $\kappa(A)$ is larger than about $10^4$.
Conversely, the convergence speed of \texttt{DE} is lower than that of \texttt{GJ2} for small $\kappa(A)$.
When $\alpha$ is a unit fraction and $\kappa(A)$ is large, the convergence speed of \texttt{GJ1} is higher than that of each of the other algorithms except for $\alpha=0.5$.
From Figure \ref{fig:2}, we can select the fastest-converging quadrature formula.
For example, when we compute $A^\alpha$ where $\alpha=0.5$ and $\kappa(A)=10^{10}$, \texttt{DE} is the fastest to converge among the three.

\section{Numerical experiments}\label{sec:4:experiments}
The numerical experiments were carried out by using Julia 1.5.1 on a Core-i7 (3.6 GHz) CPU with 16 GB RAM.
The IEEE double-precision arithmetic is used unless otherwise stated.
For arbitrary precision arithmetic, the programs use the \texttt{BigFloat} data type of Julia, which gives roughly 77 significant decimal digits.
Abscissas and weights in the GJ quadrature are computed with \texttt{FastGaussQuadrature.jl}.\footnote{\url{https://github.com/JuliaApproximation/FastGaussQuadrature.jl}}
The test matrices are listed in Table \ref{tab:matrices}.
The programs for these experiments are available on Github.\footnote{\url{https://github.com/f-ttok/article-powmde/}}

To avoid computing extremely large or extremely small values, test matrices are scaled before the computation of $A^\alpha$.
Specifically, $A^\alpha = c^{-\alpha}(cA)^{\alpha}$ is computed with $c=1/\sqrt{\sigmamax\sigmamin}$, where $\sigmamax$ and $\sigmamin$ are the maximum and the minimum singular values of $A$, respectively.
In this section, we denote the scaled matrix $cA$ as $\tilde{A}$ for convenience.
When $A$ is an HPD matrix, the scaling parameter $c$ can be chosen specialized to \texttt{GJ2} so that the error of \texttt{GJ2} is small \cite{aceto_rational_2019}, see, Appendix B.
We denote this algorithm \texttt{GJ2pre} because we can regard selecting $c$ appropriately as a preconditioning to make a good use of \texttt{GJ2}.

\begin{table}\centering
  \caption{%
    Test matrices for the experiments.
    All matrices are real.
    The term SPD stands for symmetric positive definite.
    For \texttt{pores\_1}, we computed $(-A)^{\alpha}$ because all the eigenvalues of \texttt{pores\_1} lie in the left half plane.
    For \texttt{cell1}, \texttt{TSOPF\_RS\_b9\_c6}, and \texttt{circuit\_3} we computed $(A+10^{-8}I)^\alpha$, $(A+40.35I)^\alpha$, and $(A+3I)^\alpha$, respectively, because they have real negative eigenvalues.
  }
  \label{tab:matrices}
  \begin{tabular}{c|llll}\hline
    Test & Matrix & $n$ & Condition number & Properties\\\hline
    1,2 & \texttt{ex5} \cite{davis_university_2011} & 27 & $6.6\times 10^7$ & SPD\\
    & \texttt{pores\_1} \cite{davis_university_2011} & 30 & $1.8\times 10^6$ & Nonsymmetric\\
    \hline
    3 & \texttt{nos4} \cite{davis_university_2011} & 100 & $1.6\times 10^3$ & SPD\\
    & \texttt{bcsstk04} \cite{davis_university_2011} & 132 & $2.3\times 10^6$ & SPD\\
    & \texttt{lund\_b} \cite{davis_university_2011} & 147 & $3.0\times 10^4$ & SPD\\
    \hline
    4 & \texttt{SPD\_well} & 100 & $1.0\times 10^2$ & SPD\\
    & \texttt{SPD\_ill} & 100 & $1.0\times 10^7$ & SPD\\
    & \texttt{NS\_well} & 100 & $1.0\times 10^2$ & Nonsymmetric\\
    & \texttt{NS\_ill} & 100 & $1.0\times 10^7$ & Nonsymmetric\\
    \hline
    5 & \texttt{s2rmt3m1} \cite{davis_university_2011} & 5489 & $2.5\times 10^8$ & SPD\\
     & \texttt{fv3} \cite{davis_university_2011} & 9801 & $2.0\times 10^3$ & SPD\\
     & \texttt{poisson200} & 40000 & $1.6\times 10^4$ & SPD\\
     & \texttt{cell1} \cite{davis_university_2011} & 7055 & $8.5\times 10^8$ & Nonsymmetric\\
     & \texttt{TSOPF\_RS\_b9\_c6} \cite{davis_university_2011} & 7224 & $5.6\times 10^4$ & Nonsymmetric\\
     & \texttt{circuit\_3} \cite{davis_university_2011} & 12127 & $1.1\times 10^2$ & Nonsymmetric\\
    \hline
  \end{tabular}
\end{table}

\subsection{Test 1: Selected intervals in our algorithms}
In this test, we check that the subroutine \texttt{GetInterval} truncates the infinite interval appropriately.
We consider the computation of the square root of test matrices according to the following procedure.
First, we computed the reference solution $A^{1/2}$ by using the Denman--Beavers (DB) iteration \cite[p.\ 148]{higham_functions_2008} with arbitrary precision.
Next, we computed the finite interval in double precision.
We set $\epsilon = (c\rho(A))^\alpha 10^{-7}, (c\rho(A))^\alpha 10^{-14}$ where $c = 1/\sqrt{\sigmamax\sigmamin}$ so that the relative error in the 2-norm is smaller than $10^{-7}$ or $10^{-14}$.
The norms $\|\tilde{A}\|_2$ and $\|\tilde{A}^{-1}\|_2$ were computed in three ways.
First, the computation was performed using the function \texttt{svdvals} in Julia, which computes all the singular values of a matrix.
For the second method, the computation was performed to achieve three-digit accuracy by using \texttt{Arpack.jl},\footnote{\url{https://github.com/JuliaLinearAlgebra/Arpack.jl}} which computes the eigenvalues and singular values on the basis of the Arnoldi method.
For the third method, the computation was performed using \texttt{Arpack.jl} as in the second method, but we set $\epsilon = \tilde{\epsilon}/(1 + 1/(1-10^{-3}))$ for $\tilde{\epsilon} = 10^{-7}, 10^{-4}$.
The third method is denoted as \texttt{Arpackmod}.
Then, we computed $A^{1/2}$ by using the $m$-point DE formula with arbitrary precision to avoid the error caused by matrix inversion.
For reference, we also computed $A^{1/2}$ with the wider finite interval $[-6,6]$.

As a first result, the right end values $r$ of the finite intervals $[l,r]$ computed by \texttt{GetInterval} are listed in Table \ref{tab:finite_interval}.
The data in Table \ref{tab:finite_interval} show that there is only a slight difference between the interval with accurate norms (\texttt{svdvals}) and that with rough norms (\texttt{Arpack}).
If we consider the error of the norms (\texttt{Arpackmod}), \texttt{GetInterval} provides a slightly wider interval than that of \texttt{svdvals}.
Hence, the error of the norms can be neglected.
\begin{table}
  \centering
  \caption{Right end values $r$ of finite intervals $[l,r]$ selected by Algorithm \ref{alg:1}}
  \label{tab:finite_interval}
  \begin{tabular}{lc|rrr}
    Matrix & $\epsilon$ & \texttt{svdvals} & \texttt{Arpack} & \texttt{Arpackmod}\\
    \hline
    \texttt{ex5} & $10^{-7}$ & 4.0189456993 & 4.0189454235 & 4.0501871836\\
    \texttt{ex5} & $10^{-14}$ & 4.5713980347 & 4.5713979514 & 4.5895014861\\
    \texttt{pores\_1} & $10^{-7}$ & 3.9825518994 & 3.9825518993 & 4.0149323090\\
    \texttt{pores\_1} & $10^{-14}$ & 4.5506094014 & 4.5506093993 & 4.5690896458\\
  \end{tabular}
\end{table}

Next, in Figure \ref{fig:app-1}, we show the convergence histories of the DE formula.
Figure \ref{fig:app-1} shows that the approximations achieved the required accuracy in all cases.
Therefore, the truncation error of Algorithm \ref{alg:1} is smaller than the given tolerance.
As expected from the results listed in Table \ref{tab:finite_interval}, the computational results are accurate regardless of the accuracy of the parameters.
In addition, the DE formula with selected finite intervals converges faster than the DE formula with the wide interval.
In conclusion, \texttt{GetInterval} determines finite intervals appropriately, and we can use the roughly computed parameters.
\begin{figure}
  \centering
  \includegraphics[width=\linewidth]{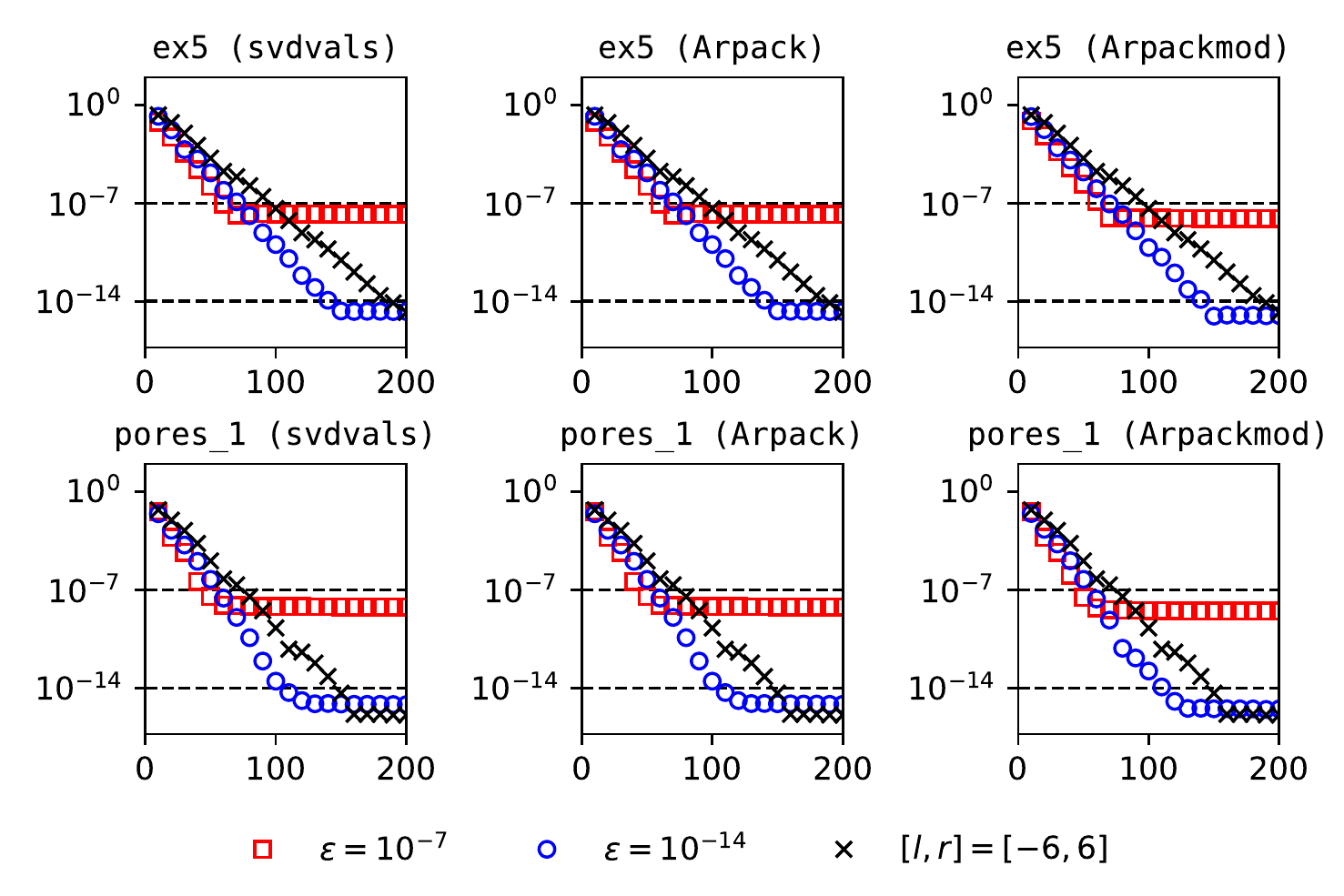}
  \caption{Convergence histories of the DE formula for $A^{1/2}$ in Test 1.
    The vertical axes show the relative error in the 2-norm, and the horizontal axis shows the number of abscissas.
    The labels in each graph indicate the test matrix and the method for computing norms $\|\tilde{A}\|_2$, $\|\tilde{A}^{-1}\|_2$.
    The notation \texttt{svdvals} indicates that the parameters are computed by using \texttt{svdvals}, \texttt{Arpack} indicates that the norms are computed approximately by using \texttt{Arpack.jl}, and \texttt{Arpackmod} indicates that the norms are computed by using \texttt{Arpack.jl}, but $\epsilon$ is set so that the error of norms does not affect the estimate of the truncation error.
    The values $\epsilon$ in the legend represent the tolerance of relative error, and ``$[l,r]=[-6,6]$'' indicates that the approximation is computed with a wide interval $[-6,6]$.
  }
  \label{fig:app-1}
\end{figure}

\subsection{Test 2: Accuracy of Algorithm 2}
In this test, we checked the accuracy of Algorithm \ref{alg:2} by computing $A^\alpha$ $(\alpha=0.2, 0.5, 0.8)$.
The reference solution $A^{\alpha}$ was computed by using the DB iteration and the inverse Newton method \cite[Alg.\ 7.14]{higham_functions_2008} with arbitrary precision.
The reference solution $A^{0.8}$ was computed via $A^{0.8} = (A^{1/5})^4$.
Then, we computed $\tilde{A}^\alpha$ by using Algorithm \ref{alg:2} with arbitrary precision.
In Algorithm \ref{alg:2}, the parameter $\epsilon$ is set so that the relative error in the 2-norm is smaller than $10^{-7}$ and $10^{-14}$.

The relative error and its estimate in Algorithm \ref{alg:2} are shown in Figure \ref{fig:3-1}.
Figure \ref{fig:3-1} shows that Algorithm \ref{alg:2} achieved the required accuracy for all cases.
The behavior of the errors shows that Algorithm \ref{alg:2} controlled the truncation error in \eqref{eq:relerr}, and the behavior of the estimates shows that Algorithm \ref{alg:2} controlled the discretization error in \eqref{eq:trap_err_estimation}.

\begin{figure}
  \centering
  \includegraphics[width=\linewidth]{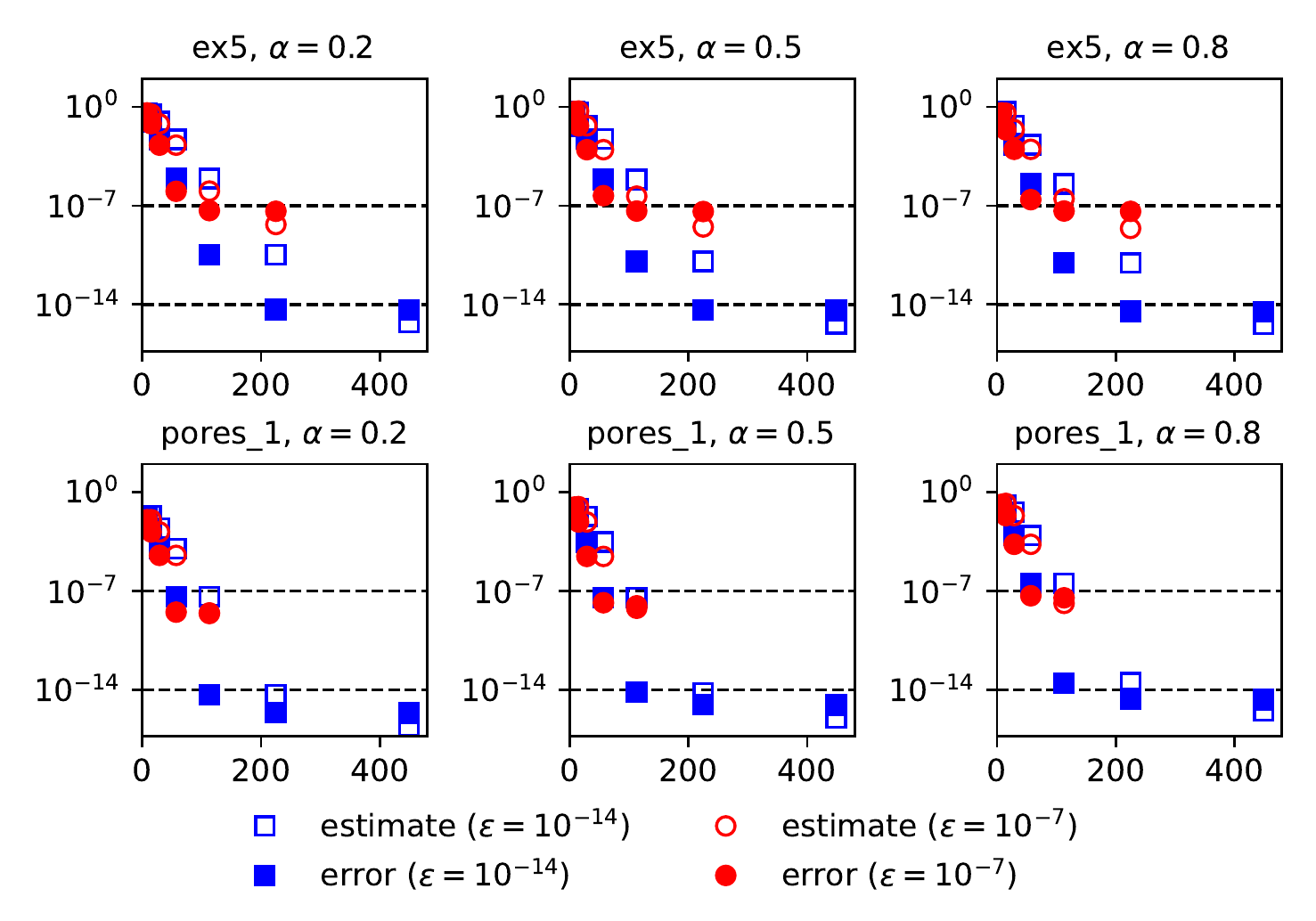}
  \caption{Relative error and its estimate in Algorithm \ref{alg:1}.
    The vertical axes show the relative error in the 2-norm, and the horizontal axes show the number of abscissas $m_s$.
    The values $\epsilon$ represent the tolerance for the relative error.
  }
  \label{fig:3-1}
\end{figure}

\subsection{Test 3: Checking the convergence speed of the DE formula}
In this test, we checked the convergence speed of the DE formula for the fractional power of symmetric positive definite (SPD) matrices.
We computed $\tilde{A}^{1/2}$ by using the DE formula (Algorithm \ref{alg:1}) with $\epsilon=2^{-53}(c\rho(\tilde{A}))^\alpha$ in double-precision arithmetic.
Subsequently, we computed $\tilde{\lambda}_{\mathrm{max}}^{1/2}$, i.e., the square root of the maximum eigenvalue of $\tilde{A}$, by using the DE formula with an integral interval identical to that of the DE formula for $\tilde{A}^{1/2}$.
The convergence histories of the DE formula are shown in Figure \ref{fig:3-2}.
In addition, we illustrated the estimates of the convergence rate \eqref{eq:rate_DE} in Figure \ref{fig:3-2}.

Figure \ref{fig:3-2} shows that the convergence of the DE formula for $\tilde{A}^{1/2}$ is almost equal to that of the DE formula for $\tilde{\lambda}_{\mathrm{max}}^{1/2}$.
Furthermore, the convergence rates of the quadrature formulas are approximately equal to the estimates.
These results support our convergence rate analysis, and hence, we can predict the convergence of the DE formula for $\tilde{A}^\alpha$ from the behavior of the DE formula for $\tilde{\lambda}_{\mathrm{max}}^\alpha$.
\begin{figure}
  \centering
  \includegraphics[width=\linewidth]{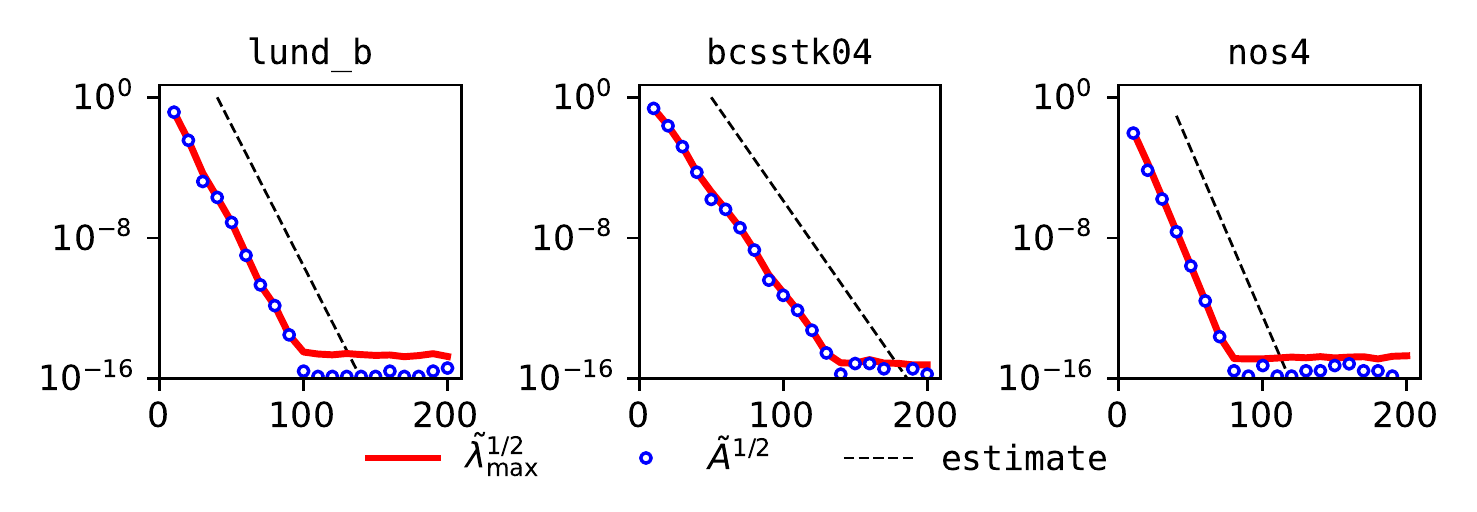}
  \caption{Convergence histories of the DE formula for the square root of the test matrices and its maximum eigenvalues. The vertical axes show the relative error in the 2-norm, and the horizontal axes show the number of abscissas.}
  \label{fig:3-2}
\end{figure}

\subsection{Test 4: Comparison of the DE formula with the Gauss--Jacobi quadrature}
We compared the convergence of the DE formula with that of the GJ quadrature.
For this test, we generated four test matrices:
\begin{enumerate}
  \item We generated two SPD matrices (\texttt{SPD\_well} and \texttt{SPD\_ill}).
  \begin{enumerate}
    \item[1-1.] We generated an orthogonal matrix $Q$ by taking the orthogonal vector of the QR decomposition of a random $100\times 100$ matrix.
    \item[1-2.] We generated a diagonal matrix $D = \mathrm{diag}(d_1,\dots,d_n)$ whose diagonal elements were derived from the geometric sequence $\{d_i\}_{i=1,\dots,100}$, where $d_1=\kappa^{-1/2}$ and $d_{100}=\kappa^{1/2}$ for $\kappa=10^2$.
    \item[1-3.] $A_{\mathrm{SPD\_well}}=QDQ^\top$
    \item[1-4.] We repeated Step 1-2 and 1-3 with the setting $\kappa=10^7$ and generated $A_{\mathrm{SPD\_ill}}$.
  \end{enumerate}
  \item We generated two nonsymmetric matrices (\texttt{NS\_well} and \texttt{NS\_ill}).
  \begin{enumerate}
    \item[2-1.] We generated a random $100\times 100$ matrix $R$.
    \item[2-2.] We generated $A_{\mathrm{NS\_well}} = \exp(cR)$ with $c=c_{\mathrm{well}} \approx 8.633\times 10^{-2}$ so that $\kappa(A_{\mathrm{NS\_well}}) \approx 10^2$.
    The matrix exponential was computed by using \texttt{exp} function in Julia.
    \item[2-3.] We scaled $A_{\mathrm{NS\_well}}$ so that the product of the maximum singular value and the minimum singular value becomes 1.
    \item[2-4.] We repeated Step 2-2 and 2-3 and generated $A_{\mathrm{NS\_ill}}$ with the setting $c=c_{\mathrm{ill}} \approx 3.029\times 10^{-1}$ so that $\kappa(A_{\mathrm{NS\_ill}}) \approx 10^7$.
  \end{enumerate}
\end{enumerate}
Then, we computed $A^\alpha$ for $\alpha=0.2,0.5,0.8$. The reference solutions are computed as in Test 1.
The convergence histories are shown in Figure \ref{fig:3-3}.
\begin{figure}
  \centering
  \includegraphics[width=\linewidth]{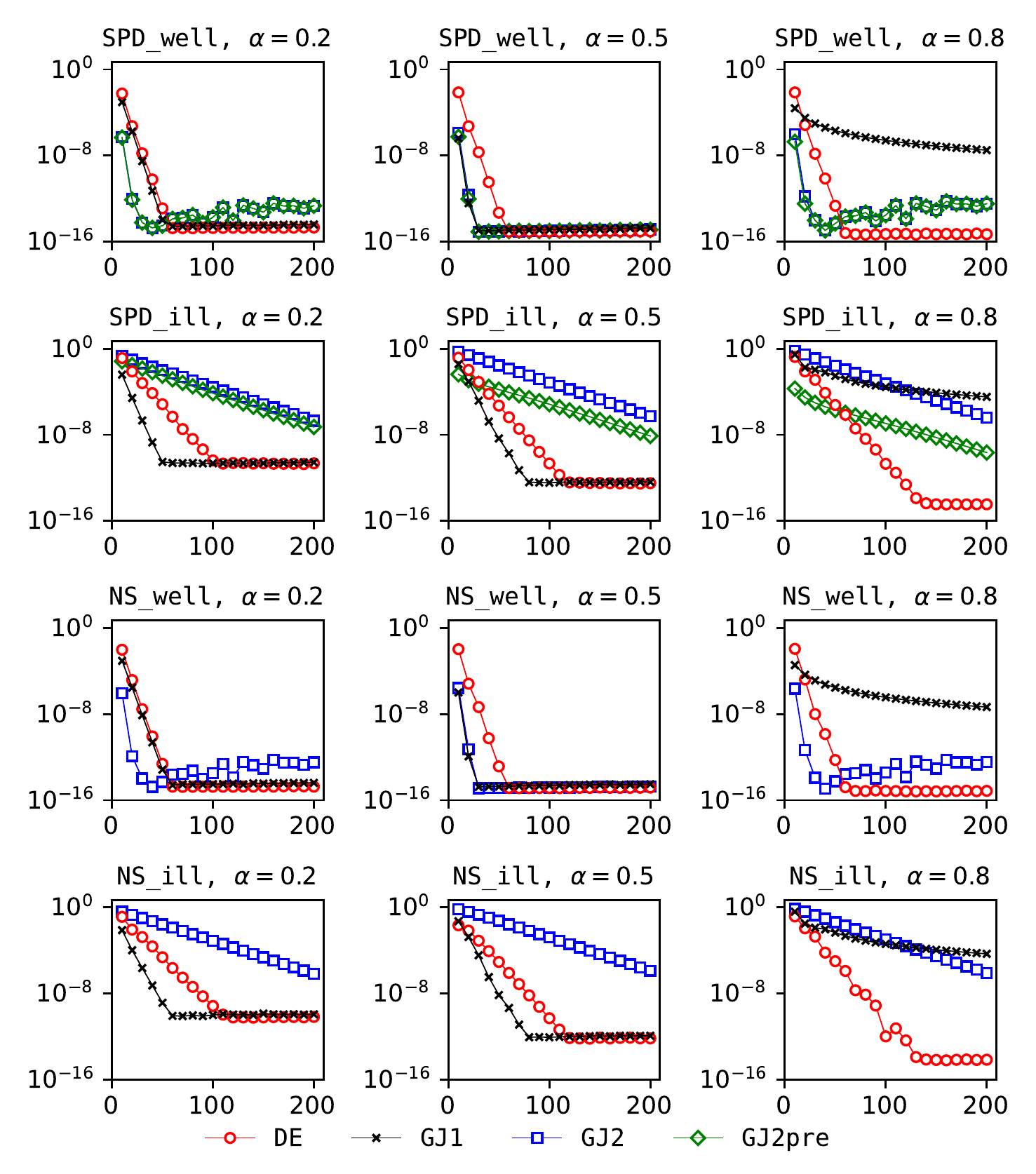}
  \caption{Convergence histories of the four or three quadrature formulas: The $m$-point DE formula (\texttt{DE}), the GJ quadrature for \eqref{eq:ir_cardoso} (\texttt{GJ1}), the GJ quadrature for \eqref{eq:ir_fasi} (\texttt{GJ2}), and preconditioned \texttt{GJ2} for HPD matrices (\texttt{GJ2pre}). The vertical axes show the relative error in the Frobenius norm, and the horizontal axes show the number of abscissas.}
  \label{fig:3-3}
\end{figure}
Figure \ref{fig:3-3} shows that the DE formula works well in all cases.
In addition, when $\alpha$ is a non-unit fraction ($\alpha=0.8$) and $\kappa(A)$ is large (\texttt{SPD\_ill} and \texttt{NS\_ill}), \texttt{DE} is the fastest to converge among the four or three quadrature formulas.

\subsection{Test 5: Computational time for \texorpdfstring{$A^\alpha\bmb$}{the action of the matrix fractional power}}
In this test, we measured the computational time of quadrature-based algorithms for computing $A^\alpha\bmb$.
The test matrix \texttt{poisson200} was generated by $A = L_{200}\otimes I_{200} + I_{200}\otimes L_{200}$, where $L_{200} = \mathrm{tridiag}(-1, 2, -1) \in \bbR^{200\times 200}$, $I_{200}$ is an identity matrix of size 200, and $\otimes$ is the Kronecker product.
The vector $\bmb$ is set to a random vector with norm 1.
The test procedure is as follows:
\begin{description}
  \item[SPD matrices] For SPD matrices, we considered five quadrature-based algorithms: \texttt{GJ1}, \texttt{GJ2}, \texttt{GJ2pre}, \texttt{DE}, and the algorithm based on the Cauchy integral presented in \cite[Sect. 3]{hale_computing_2008}, which is denoted by \texttt{Cauchy}.
  The algorithm \texttt{Cauchy} requires the computation of the complete elliptic integral of the first kind and the Jacobi elliptic functions.
  We implemented them by using \texttt{Elliptic.jl}\footnote{\url{https://github.com/nolta/Elliptic.jl}} for the real argument.
  For the complex argument, we used the relations in \cite[Eqs.\ (16.21.1) -- (16.21.4)]{abramowitz_handbook_1964}.

  First, we estimated the number of abscissas $m$ such that the absolute error $\|A^\alpha\bmb-\bmx\|_2$ is smaller than or equal to $10^{-6}$.
  For \texttt{GJ1}, \texttt{GJ2}, and \texttt{DE}, we can predict the convergence of these algorithms by the scalar convergence on the extreme eigenvalue of $\tilde{A}$.
  Hence, we counted the number of abscissas in the scalar case and used the same number of abscissas for the matrix case.
  For \texttt{GJ2pre}, we used the upper bound on the error, which can be computed with the extreme eigenvalues of $A$, derived in \cite{aceto_rational_2019}.
  The extreme eigenvalues were computed with three-digit accuracy using \texttt{Arpack.jl}.
  For \texttt{Cauchy}, the number of abscissas was estimated by applying \texttt{Cauchy} to a $2 \times 2$ matrix whose eigenvalues were the extreme eigenvalues of $\tilde{A}$ as in \cite{fasi_computing_2018}.
  When the estimated number of the abscissas was greater than 1000, we set the number of abscissas to 1000.
  Then, we computed $A^\alpha\bmb$.
  The linear systems in the integrand were solved using a sparse direct linear solver in \texttt{SuiteSparse.jl}.\footnote{\url{https://github.com/JuliaLinearAlgebra/SuiteSparse.jl}}

  \item[Nonsymmetric matrices] For general matrices, we considered three adaptive quadrature algorithms, namely, Algorithm \ref{alg:2} (denoted by \texttt{DE}) and two algorithms based on the GJ quadrature for \eqref{eq:ir_cardoso} and \eqref{eq:ir_fasi} (denoted by \texttt{GJ1} and \texttt{GJ2}, respectively).
  Algorithms \texttt{GJ1} and \texttt{GJ2} compute approximations by using the $(2^{s}m_0)$-point GJ quadrature for $s=0,1,2,\dots$.
  As in \cite[Alg.\ 5.2]{cardoso_computation_2012}, the error of the GJ quadrature is estimated by
  \begin{align}
    \left\lVert \bmx_{2m} - A^\alpha\bmb \right\rVert_2
    \le \left\lVert \bmx_{2m} - \bmx_{m} \right\rVert_2,
  \end{align}
  where $\bmx_m$ is the approximation of the $m$-point GJ quadrature.
  In contrast to the algorithm of \cite[Alg.\ 5.2]{cardoso_computation_2012}, our algorithms did not compute the Schur decomposition and the matrix square root because $A$ was sparse.

  As for SPD matrices, we set the tolerance so that $\|A^\alpha\bmb - \bmx\|_2 \le 10^{-6}$.
  For all algorithms, we set the initial number of abscissas $m_0$ as 8.
  The extreme singular values for scaling the matrices were computed via \texttt{Arpack.jl}, and the linear systems were solved using \texttt{SuiteSparse.jl}.
  The algortihms were stopped when the number of the evaluations of the integrand exceeded 1000.
\end{description}

Table \ref{tab:cputime_spd} lists the computational times and the number of evaluations of the integrand of the quadrature-based algorithms for SPD matrices.
This computational time includes the time for computing the extreme singular values and the extreme eigenvalues for scaling the test matrices.
When $\alpha$ is a non-unit fraction ($\alpha = 0.8$) and $A$ is ill-conditioned matrix (\texttt{s2rmt3m1}), \texttt{DE} is the fastest. Otherwise, \texttt{GJ1} or \texttt{GJ2pre} is the fastest.
Although \texttt{Cauchy} requires a small number of abscissas, it is not the fastest in this test because of the use of complex arithmetic.
Table \ref{tab:cputime_ns} lists the times of the algorithms for nonsymmetric matrices.
For well-conditioned matrix (\texttt{circuit\_3}), \texttt{GJ2} is the fastest, otherwise, \texttt{DE} is the fastest.
The algorithm \texttt{DE} seems to be more robust than \texttt{GJ1} and \texttt{GJ2} because of the efficiency of the discretization error estimation.

\begin{table}[p]
  \centering
  \rotatebox{90}{
    \begin{minipage}{0.8\textheight}
      \caption{Comparison of quadrature-based algorithms for SPD matrices in terms of CPU time (in seconds) and the number of evaluations of the integrand (in parentheses).
      If the number of evaluations exceeds 1000, the corresponding results are underlined.
      The results of the fastest algorithm among the five algorithms are written in bold font.}
      \begin{tabular}{ll|rr|rr|rr|rr|rr}
        Matrix & $\alpha$ & \multicolumn{2}{|c}{\texttt{GJ1}} & \multicolumn{2}{|c}{\texttt{GJ2}} & \multicolumn{2}{|c}{\texttt{GJ2pre}} & \multicolumn{2}{|c}{\texttt{DE}} & \multicolumn{2}{|c}{\texttt{Cauchy}}\\
        \hline
        \texttt{s2rmt3m1} & 0.2 & \textbf{0.59} & (\textbf{39}) & 7.20 & (511) & 6.38 & (451) & 1.07 & (73) & 1.30 & (21)\\
        \texttt{s2rmt3m1} & 0.8 & \underline{14.37} & (\underline{1000}) & 10.55 & (731) & 6.98 & (485) & \textbf{1.52} & (\textbf{95}) & 1.90 & (31)\\
        \texttt{fv3} & 0.2 & 0.30 & (23) & 0.30 & (24) & \textbf{0.28} & (\textbf{23}) & 0.36 & (30) & 0.41 & (11)\\
        \texttt{fv3} & 0.8 & 2.10 & (193) & 0.32 & (26) & \textbf{0.26} & (\textbf{21}) & 0.38 & (29) & 0.38 & (10)\\
        \texttt{poisson200} & 0.2 & \textbf{1.43} & (\textbf{25}) & 1.82 & (41) & 1.69 & (38) & 1.59 & (33) & 2.01 & (13)\\
        \texttt{poisson200} & 0.8 & 14.34 & (352) & 1.97 & (45) & \textbf{1.55} & (\textbf{34}) & 1.65 & (33) & 1.91 & (12)\\
      \end{tabular}
      \label{tab:cputime_spd}
    \end{minipage}
  }
  \rotatebox{90}{
    \begin{minipage}{0.8\textheight}
      \caption{Comparison of quadrature-based algorithms for nonsymmetric matrices in terms of CPU time (in seconds) and the number of evaluations of the integrand (in parentheses).
      If the number of evaluations exceeds 1000, the corresponding results are underlined.
      The results of the fastest algorithm among the three algorithms are written in bold font.}
      \begin{tabular}{ll|rr|rr|rr}
        Matrix & $\alpha$ & \multicolumn{2}{|c}{\texttt{GJ1}} & \multicolumn{2}{|c}{\texttt{GJ2}} & \multicolumn{2}{|c}{\texttt{DE}}\\
        \hline
        \texttt{cell1} & 0.2 & 4.00 & (248) & \underline{15.86} & (\underline{1016}) & \textbf{3.55} & (\textbf{225})\\
        \texttt{cell1} & 0.8 & \underline{15.88} & (\underline{1016}) & \underline{15.87} & (\underline{1016}) & \textbf{3.60} & (\textbf{225})\\
        \texttt{TSOPF\_RS\_b9\_c6} & 0.2 & 0.93 & (120) & 1.93 & (248) & \textbf{0.88} & (\textbf{113})\\
        \texttt{TSOPF\_RS\_b9\_c6} & 0.8 & \underline{7.82} & (\underline{1016}) & 1.92 & (248) & \textbf{0.87} & (\textbf{113})\\
        \texttt{circuit\_3} & 0.2 & 2.36 & (120) & \textbf{0.81} & (\textbf{24}) & 1.36 & (57)\\
        \texttt{circuit\_3} & 0.8 & 2.21 & (120) & \textbf{0.91} & (\textbf{24}) & 1.33 & (57)\\
      \end{tabular}
      \label{tab:cputime_ns}
    \end{minipage}
  }
\end{table}

\section{Conclusion}\label{sec:5:conclusion}
In this work, we considered the DE formula to compute $A^\alpha$.
To utilize the DE formula, we proposed a method of truncating the infinite interval based on the analysis of the truncation error specialized for $A^\alpha$.
Based on the truncation error analysis, we presented two algorithms, including an adaptive quadrature algorithm.
We analyzed the convergence rate of the DE formula for HPD matrices and found that when $\alpha$ is a non-unit fraction and $\kappa(A)$ is large, the DE formula converges faster than the GJ quadrature.

We performed five numerical tests.
The results of the first two tests showed that our algorithms achieved the required accuracy.
The result of the third test showed that our estimate of the convergence rate is appropriate.
The results of the fourth and fifth tests showed that our algorithms work well, especially when the algorithms based on the GJ quadrature converge slowly.

In the future, we plan to consider the practical performance of our algorithm with parallel computing when applied to large practical problems.

\section*{Acknowledgments}
We would like to thank Prof. Ken'ichiro Tanaka for his valuable comments on the convergence rate analysis of the DE formula.
We are grateful to the two anonymous referees for the careful reading and the comments that substantially enhanced the quality of the manuscript.
This work was supported by JSPS KAKENHI Grant Numbers 18J22501 and 18H05392.
The authors would like to thank Enago (www.enago.jp) for the English language review.

\appendix
\section{On the behavior of \texorpdfstring{$|\fde(x+\rmi y,\lambda)|$}{|fde(x+iy)|} \texorpdfstring{as $x\to \pm\infty$}{}} \label{sec:app-calc}
In this section, we consider the asymptotic behavior of $\fde$ defined in \eqref{eq:fde}.
We show that there exist constants $c_l, c_r, l$, and $r$ such that for any $y$ satisfying $|y| < d_0(\lambda) - \delta\; (<\pi/2)$,
\begin{align}
  &|\fde(x+\rmi y)| < c_l \exp\left(-\frac{\alpha\pi \cos(d_0(\lambda)-\delta)}{2} \sinh(|x|)\right) \cosh x  \quad (x < l),
  \label{eq:app-absfdel}\\
  &|\fde(x+\rmi y)| < c_r \exp\left(-\frac{(1-\alpha)\pi \cos(d_0(\lambda)-\delta)}{2} \sinh x\right) \cosh x
  \quad (x > r).
  \label{eq:app-absfder}
\end{align}

From $|\cosh(x+\rmi y)|=\sqrt{\cosh^2 x + \cos^2 y-1}$, it follows that
\begin{align} \label{eq:app-abscosh}
  |\cosh(x+\rmi y)| \le \cosh x.
\end{align}
Besides, for $x<0$,
\begin{align}
  \left|\exp\left(\frac{\alpha\pi}{2}\sinh(x+\rmi y)\right)\right|
  \le \exp\left(-\frac{\alpha\pi}{2} \sinh (|x|) \cos(d_0(\lambda)-\delta) \right).
  \label{eq:app-absexpasinh}
\end{align}
Because $\fde$ is bounded in $\calD_{d_0(\lambda) - \delta}$, there exists a constant $c_1$ that satisfies 
\begin{align}\label{eq:app-absexpsinhlmdl}
  \left|\exp\left(\frac{\pi}{2}\sinh(x+\rmi y)\right) + \lambda \right| > c_1.
\end{align}
By combining \eqref{eq:app-abscosh}, \eqref{eq:app-absexpasinh}, \eqref{eq:app-absexpsinhlmdl}, we get \eqref{eq:app-absfdel}.

Next, it is true that
\begin{align}\label{eq:app-absexpasinhr}
  \left|\exp\left(\frac{\alpha\pi}{2}\sinh(x+\rmi y)\right)\right|
  = \exp\left(\frac{\alpha\pi}{2} \sinh(x)\cos(y)\right).
\end{align}
In addition, there exists a constant $c_2$ such that for $x>0$ satisfying $\exp (\pi \sinh x \cos(d_0(\lambda) - \delta)/2) > \lambda$, 
\begin{align}
  &\left|\exp\left(\frac{\pi}{2}\sinh(x+\rmi y)\right) + \lambda \right|
  \ge \exp\left(\frac{\pi}{2}\sinh x\cos y\right) - \lambda\\
  &\quad \ge c_2 \exp\left(\frac{\pi}{2}\sinh x\cos y\right).
\end{align}
Hence, 
\begin{align}\label{eq:app-absexpsinhlmdr}
  &\frac{|\exp(\alpha\pi\sinh(x+\rmi y)/2)|}{|\exp(\pi\sinh(x+\rmi y)/2) + \lambda|}
  \le \frac{1}{c_2} \exp\left(-\frac{(1-\alpha)\pi}{2}\sinh(x)\cos(y) \right)\\
  &\quad \le \frac{1}{c_2} \exp\left(-\frac{(1-\alpha)\pi \cos(d_0(\lambda)-\delta)}{2}\sinh(x)\right).
\end{align}

By combining \eqref{eq:app-abscosh}, \eqref{eq:app-absexpasinhr}, and \eqref{eq:app-absexpsinhlmdr}, we get  \eqref{eq:app-absfder}.

\section{A preconditioning technique of \texttt{GJ2} for HPD matrices}
The study \cite{aceto_rational_2019} considers the computation of fractional powers of a positive self-adjoint operator $\calL$ by using the GJ quadrature for the integral
\begin{align} \label{eq:ir_aceto}
  \calL^{-\alpha}
    = \frac{2\sin(\alpha\pi)\tau^{1-\alpha}}{\pi}
      \int_{-1}^{1} (1-t)^{-\alpha} (1+t)^{\alpha-2} \left(\tau\frac{1-t}{1+t}\calI + \calL \right) \dd{t},
\end{align}
where $\alpha\in(0,1)$, $\calI$ is an identity operator, and $\tau>0$ is a parameter \cite[Eq. (2)]{aceto_rational_2019}.
When $\calL$ is an HPD matrix, we can rewrite \eqref{eq:ir_aceto} to \eqref{eq:ir_fasi}, which is the integral used in \texttt{GJ2}, by substituting $\calL^{-1} = A$ and $\tau=1$.

One of the contributions of \cite{aceto_rational_2019} is proposing a method for selecting $\tau$ to reduce the error.
The proposed selection is
\begin{align} \label{eq:selection_tau}
  \tau = 
  \begin{cases}
    \tau^-(m) & (m < \bar{m}),\\
    \tau^+(m) & (m \ge \bar{m}),
  \end{cases}
\end{align}
where $m$ is the number of abscissas,
\begin{align}
  & \bar{m} = \frac{\alpha}{2\sqrt{2}} \left(\log(\rme^2 \kappa)\right)^{1/2} \kappa^{1/4} \qquad \left(\kappa = \frac{\mumax}{\mumin}\right),\\
  & \tau^-(m) = \mumin \left(\frac{\alpha}{2\rme m}\right)^2 \exp\left(2\rmW\left(\frac{4\rme m^2}{\alpha^2}\right) \right),\\
  & \tau^+(m) = \left(
    -\frac{\alpha \mumax^{1/2}\log(\kappa)}{8m}
    + \sqrt{\left(\frac{\alpha \mumax^{1/2}\log(\kappa)}{8m}\right)^2 + (\mumax\mumin)^{1/2}}
    \right)^2,
  \end{align}
$\rmW$ is the Lambert W function, and $\mumax$ and $\mumin$ are the maximum and the minimum eigenvalue of $\calL$, respectively.
One can apply the preconditioning technique to \texttt{GJ2}, which computes $A^\alpha$ for $\alpha \in (0,1)$, by substituting $\mumax = 1/\lambdamin, \ \mumin = 1/\lambdamax$, where $\lambdamax$ and $\lambdamin$ are the maximum and the minimum eigenvalue of $A$, and computing $A^\alpha = \tau^{-\alpha} (\tau A)^\alpha$.


\end{document}